\documentclass{amsart}
\usepackage[utf8]{inputenc}

\usepackage{tikz,floatrow}

\usepackage{amsmath,amssymb,mathabx,comment}
\usepackage{hyperref}

\usepackage{amsthm}
\newtheorem{thm}{Theorem}
\newtheorem{definition}{Definition}[section]
\newtheorem{lemma}[definition]{Lemma}
\newtheorem{prop}[definition]{Proposition}
\newtheorem{rmk}[definition]{Remark}
\newtheorem{example}[definition]{Example}

\newcommand{\diag}{\text{diag}}
\newcommand{\spec}{\text{Spec}}
\newcommand{\lcm}{\text{lcm}}

\def\la{\langle}
\def\ra{\rangle}

\begin{document}

\title{On the original Ulam's problem and its quantization}
\date{30 May 2025}
\author{Changguang Dong$^*$, Jing Zhou$^\dagger$}
\thanks{$^*$ Chern Institute of Mathematics and LPMC, Nankai University, Tianjin, China. Email: dongchg@nankai.edu.cn.
$^\dagger$ School of Sciences, Great Bay University, Dongguan, China. Email: zhoujing@gbu.edu.cn.\\
MSC classes: 37N05, 81Q50 (Primary), 35Q41, 37E10 (Secondary).}

\begin{abstract}
    We show that under general resonance the classical piecewise linear Fermi-Ulam accelerator behaves substantially different from its quantization, in the sense that the classical accelerator exhibits typical recurrence and non-escaping while the quantum version enjoys quadratic energy growth in general. We also describe a procedure to locate the escaping orbits, though exceptionally rare in the infinite-volume phase space, for the classical accelerators, which in particular include Ulam's very original proposal and the linearly escaping orbits therein in the existing literature, and hence provide a complete (modulo a null set) answer to Ulam's original question. For the quantum accelerators, we reveal under resonance the direct and explicit connection between the energy growth and the shape of the quasi-energy spectrum.
\end{abstract}

\maketitle

\section{Introduction}

In 1949 Fermi \cite{Fermi49} proposed an acceleration mechanism, as to explain for the existence of high energy particles in cosmic rays, that charged particles get repeatedly reflected by moving magnetic mirrors, arguing that the chances of head-on interactions outnumber the head-tail ones so that the particles would on average get accelerated. Later in 1961 Ulam \cite{Ulam61} extracted the effective mechanism in Fermi acceleration, i.e. a particle bounces between two infinitely heavy walls, one fixed and the other moving periodically (c.f. Fig.\ref{fig:originalFUM}). Ulam \cite{Ulam61} performed numerical simulations with a specific piecewise linear wall motion (i.e. $A=1/\sqrt2, B=\sqrt2, T=1$ in Fig.\ref{fig:originalFUM}) and then conjectured that such a model should produce \emph{escaping orbits}, i.e. those whose energy grows to infinity in time.\\ 

Since then, numerous efforts, theoretically and numerically, have been made by physicists and mathematicians in the study of the classical Fermi-Ulam accelerators and we refer to \cite{Dol-FA,GRKT12,LiLi80} and references therein for detailed survey on the subject. The application of KAM theory leads to non-acceleration results for sufficiently smooth Fermi-Ulam accelerators and their variations \cite{Dol08potential,kuOr20,LaLe91,Pu83,Pu94,Zhar2000}. On the other hand, escaping orbits appear when singularities are allowed \cite{deSD12,Zha98,Zhou21}, ambient potentials are introduced \cite{AZ15,DS09,Pu77}, and also when higher-dimensions are considered \cite{GT08nonauto,GT08ham,KZ23,STR10,Zhou20}. Classical techniques such as Birkhoff normal forms and hyperbolic dynamics play fundamental roles in the study of classical Fermi-Ulam accelerators.

For the quantum Fermi-Ulam accelerators, the parallel non-acceleration results were obtained for generic sufficiently smooth wall motions \cite{karn94,Seba90}, with the help of quantum KAM theory \cite{how89,how98}. Quantum chaos and acceleration behavior were expected both in numerics \cite{glr14,joco86} and in theory \cite{DoeRi69,Seba90} when singularities are allowed in the models. The quantum Fermi-Ulam accelerators are one of the earliest examples of quantum systems with time-dependent boundary conditions \cite{DoeRi69}, and have then inspired the development of many important techniques in the study of mesoscopic physics \cite{Bas88,cbpwc96}, quantum control \cite{lbmw03,MaNa10}, and nano-devices \cite{AKM14,PiDBMS15}, as well as experimental realizations such as cold atoms in dynamic traps \cite{HalGMDHPN10}.\\

Notably among the existing literature on the subject, the closest related to our setting are a special case of piecewise linear Fermi-Ulam accelerator \cite{Zha98} and its quantization \cite{Seba90}. In a specific piecewise linear classical Fermi-Ulam accelerator with the same particular choice of system parameters as Ulam \cite{Ulam61} (i.e. $A=1/\sqrt2, B=\sqrt2, T=1$ in Fig.\ref{fig:originalFUM}), Zharnitsky \cite{Zha98} found explicitly a family of linearly escaping orbits. However, these escaping orbits constitute a null set, leaving unanswered the behavior of the majority of orbits on the phase space which has infinite volume. Besides, it was unclear therein \cite{Zha98} if any other choices of parameters can also create acceleration. This issue of ``mystery of lucky parameters" also appears in \v{S}eba's study of the piecewise linear quantum accelerators where he obtained quantum acceleration with the choice of parameters corresponding to the special $1:1$ resonance case in our definition (c.f. Definition \ref{def:quantumresonanace} and Example \ref{example:11resonance}).\\ 

In this paper we study a general class of piecewise linear Fermi-Ulam accelerators. We show that under general resonance conditions on the system parameters, the classical accelerators exhibit typical recurrence and the escaping orbits constitute a zero-measure set on the phase space of infinite volume. Meanwhile, the quantum accelerators under general resonance enjoy quadratic energy growth in general, if not always, and the quasi-energy spectra are absolutely continuous with finitely many components. Our results cover and greatly generalize those in \cite{Zha98} and \cite{Seba90} as special cases. 

Also, we provide an explicit procedure on how to locate the escaping orbits, though exceptionally rare, in the classical Fermi-Ulam accelerators under general resonance. In particular, this includes Ulam’s very original proposal \cite{Ulam61} (which is the special resonance case in our Definition \ref{def:resonance}), and the linearly escaping orbits found by Zharnitsky therein \cite{Zha98} correspond to an accelerating segment on a very particular rational invariant curve in our language (c.f. Example \ref{example:escbdd}), and hence we provide a complete answer (modulo a null set) to Ulam’s original question. In addition, for the quantum accelerators, we reveal the direct connection between the energy growth behavior and the shape of the quasi-energy spectra with explicit formulas. 

From a technical perspective, in the classical accelerators, the rigidity of the piecewise linearity forces the dynamics to be singular and parabolic, making it impossible to apply classical tools such as KAM theory or hyperbolicity as has done in many of the existing literature.  For quantum accelerators, the time dependence, especially the delta kicks in wall motions, prevents direct application of a lot machinery for handling Schr\"odinger equations, just as in the case of the (in)famous kicked rotators. We combine several techniques from adiabatic theory, infinite ergodic theory, and quantum Floquet theory to tackle the challenges.\\

The paper is organized as follows. We dissect the presentation of results on the classical piecewise linear Fermi-Ulam models in \S\ref{sec:cFU} and the corresponding quantum models in \S\ref{sec:QFU}, and conclude with discussions on open problems in \S\ref{sec:conc}. In \S\ref{sec:cFU}, we begin with the adiabatic change of coordinates and the adiabatic normals, which serve as the Poincar\'e maps for the classical accelerators, and then we prove the typical recurrence results under general resonance. In \S\ref{sec:QFU}, we first transfer the time dependence in boundary conditions to time-dependent potentials in the Schr\"odinger equations, then we derive the Floquet wave propagators, and finally we prove the quadratic energy growth and the absolute continuity of the quasi-energy spectra under general resonance.

\section{The Classical Fermi-Ulam models}\label{sec:cFU}
We consider the original Fermi-Ulam model \cite{Ulam61}, where a point particle bounces elastically between two infinitely heavy walls, one is fixed and the other moves periodically (c.f. Fig.\ref{fig:originalFUM}). Let $l(t)$ denote the distance between two walls. In Ulam's original proposal \cite{Ulam61}, the periodic wall motion $l(t)$ is piecewise linear:
\[l(t) = 
   \begin{cases}
       B - k t & t\in[0,T)\\
       A + k (t-T) & t\in[T,2T)
   \end{cases}
\]
where $A\ne B$ are positive parameters and $k:= \dfrac{B-A}{T}$ is the slope. 

Throughout the paper, we assume without loss of generality that $B>A$ and hence $k>0$.

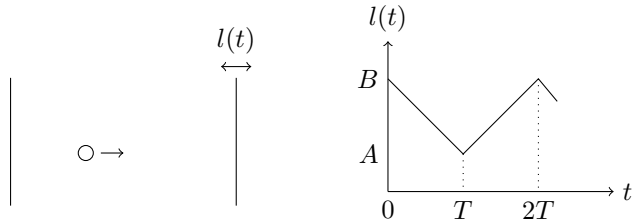
\begin{figure}[!ht]
\centering
\begin{minipage}{0.3\textwidth}
   \begin{tikzpicture}
       \draw (0,0.3)--(0,2) (3,0.3)--(3,2);
       \draw[<->] (2.8,2.15)--(3.2,2.15);
       \draw (1,1) circle (0.1);
       \draw[->] (1.2,1)--(1.5,1);
       \draw (3,2.2) node[anchor=south]{$l(t)$};
   \end{tikzpicture}   
\end{minipage}%
\hspace{20pt}%
\begin{minipage}{0.3\textwidth}
   \begin{tikzpicture}
        \draw[->] (0,0)--(3,0) node[anchor=west]{$t$};
        \draw[->] (0,0)--(0,2) node[anchor=south]{$l(t)$};
        \draw (0,1.5)--(1,0.5)--(2,1.5)--(2.25,1.2);
        \draw[dotted] (1,0.5)--(1,0) (2,1.5)--(2,0);
        \draw (0,1.5) node[anchor=east]{$B$}
              (0,0.5) node[anchor=east]{$A$}
              (0,0) node[anchor=north]{0}
              (1,0) node[anchor=north]{$T$}
              (2,0) node[anchor=north]{$2T$};
    \end{tikzpicture}
\end{minipage}
\caption{The original piecewise linear Fermi-Ulam model}
\label{fig:originalFUM}
\end{figure}

The goal of this section is to show in Theorem \ref{thm:rec} that under general resonance (as in Definition \ref{def:resonance}) the classical accelerators exhibit typical recurrence and the escaping orbits constitute a zero-measure set on the phase space of infinite volume. We also provide an explicit procedure for how to locate these exceptionally rare escaping orbits in the classical Fermi-Ulam accelerators under general resonance. In particular, this includes Ulam’s very original proposal (i.e. $A=1/\sqrt2, B=\sqrt2, T=1$) \cite{Ulam61} as a particular case of the special resonance case, and the linearly escaping orbits found by Zharnitsky therein \cite{Zha98} correspond to an accelerating segment on a very particular rational invariant curve (c.f. Example \ref{example:escbdd}).

The vital change of coordinates for the special parameters (i.e. $A=1/\sqrt2, B=\sqrt2, T=1$) as in \cite{Zha98} cannot be easily generalized, in the authors' attempt, to other choices of parameters satisfying general resonance. Instead, we begin, in \S\ref{ssec:ad}, with the adiabatic change of coordinates (c.f. Lemma \ref{lemma:adiacoord}) and the adiabatic normals (c.f. Proposition \ref{prop:adianorm}), which describe the behavior of the particle in roughly one period. Then in \S\ref{ssec:skew} we show that the adiabatic normal forms possess invariant circles (c.f. Lemma \ref{lemma:invcirc}) and then we reveal the hidden skew product structure (c.f. Proposition \ref{prop:skewprod}) on each invariant circle due to general resonance. Finally in \S \ref{ssec:rec} we prove Theorem \ref{thm:rec} of the typical recurrence results for resonant accelerators and also describe an explicit procedure to locate escaping orbits on the rational invariant circles.

\subsection{The Adiabatic Normal Forms}\label{ssec:ad}

In this section the main goal is to derive the adiabatic normal forms (c.f. Proposition \ref{prop:adianorm}) which describe the momentum change of the particle in roughly one period $\Delta t = 2T$ for large energies. 

We proceed in two stages: firstly we show, in adiabatic coordinates, that the momentum of the particle remains constant along collisions away from singularities (i.e. Lemma \ref{lemma:adiacoord}), and secondly we show, in adiabatic normals, that the momentum changes dramatically when the particle encounters singularities (i.e. Proposition \ref{prop:adianorm}). 

The techniques from adiabatic theory have been successfully applied in the study of various Fermi-Ulam accelerators \cite{deSD12,KZ23,Zhou20}.

\subsubsection{The Adiabatic Coordinates}

We start with the derivation of the adiabatic coordinates and show that the momentum (in adiabatic coordinates) of the particle remains constant along collisions away from singularities.\\

Let $t$ denote the moment of collision at the moving wall and $v$ the velocity of the particle immediately after collision at the moving wall. We denote by $f$ the collision map that sends one collision $(t_0,v_0)$ to the next one $(t_1,v_1)$. Then for $v_0$ large, we have the following equations of motion between two consecutive collisions: 
\begin{equation}\label{eq:collisoneqs}
    \begin{cases}
        v_0 (t_1 - t_0) = l(t_0) + l(t_1) \\
        v_1 = v_0 - 2\dot{l}(t_1)
    \end{cases}
\end{equation}

In spirit of the adiabatic theory, we observe that $$v_1-v_0 \approx -2\dot{l}(t_0), \; t_1-t_0 \approx \frac{2l(t_0)}{v_0},$$ which leads to the following ODE $$\frac{dv}{dt} = \frac{-v\dot{l}}{l},$$ suggesting that $I:=lv$ is almost constant. Then we update the algorithm by replacing $v$ with $I=lv$ and look for the next-order adiabatic invariant. On the other hand, the new ``time" $\theta$ needs to behave in the following way $$\theta_1-\theta_0 = \dfrac{C}{I},$$ which leads to another ODE (we aim for constant $C=1$ for simplicity) $${\theta}'\frac{2l}{v}=\frac{1}{lv},$$ giving $\displaystyle \theta(t) = \int_0^t l^{-2}(s) ds$.\\

Ideally, the Euler scheme above would terminate in finite steps since we are dealing with (piecewise) linear wall motions $l(t)$. We propose the following formulas for the adiabatic coordinates and verify in Lemma \ref{lemma:adiacoord} that this is indeed true.

We introduce the new (half) period $\displaystyle \mathcal{T}:=\int_0^T \dfrac{ds}{l(s)^2} = \dfrac{T}{AB}$. We define the adiabatic coordinates $(\theta,I)$ as follows 
\begin{equation}\label{eq:adiacoordef}
    \begin{cases}
        I(t,v) := \mathcal{T} ( l(t)v + l(t)\dot l(t) )\\
    \theta(t) := \dfrac{1}{2\mathcal{T}} \displaystyle\int_0^t \dfrac{ds}{l(s)^2}.
    \end{cases}
\end{equation}

Now we verify that $(\theta,I)$ as defined in \eqref{eq:adiacoordef} are indeed adiabatic in the sense that the new momentum $I$ is constant along collisions away from singularities. 

Let $S_0$ $(S_T)$ denote the singular collisions occurring at $t=0$ ($t=T$):
\[
   S_0 :=\{(t,v)|t=0\}, \; S_T :=\{(t,v)|t=T\}.
\]

Let $R_0$ ($R_T$) be the singularity strip bounded by $S_0$ ($S_T$) and $f(S_0)$ ($f(S_T)$) in the $(t,v)$-phase cylinder. $R_0$ ($R_T$) is, in fact, the Poincar\'e section that collects the very first collisions right after the moment $t=0$ ($t=T$) (c.f. Fig.\ref{fig:adianorm}). The name ``adiabatic" manifests itself in the following lemma. 

\begin{lemma}[Adiabatic coordinates]\label{lemma:adiacoord}
For $(t_0,v_0)\notin f^{-1}(R_0) \cup f^{-1}(R_T)$ and $v_0$ sufficiently large, the collision map $f$ in adiabatic coordinates $(\theta,I)$ as defined in \eqref{eq:adiacoordef} is given by 
\begin{equation}\label{eq:adiacoor}
     \theta_1 = \theta_0 + \frac{1}{I_0}, \; I_1 = I_0.
\end{equation}
\end{lemma}

\begin{proof}
We present the proof of the case when $t_0,t_1\in(0,T)$; the other case when $t_0,t_1\in(T,2T)$ can be obtained in a similar fashion.\\

First in the momentum direction 
\begin{align*}
    \mathcal{T}^{-1} (I_1 - I_0) 
    &= (l_1 v_1 + l_1 \dot l_1) - (l_0 v_0 + l_0 \dot l_0) \\
    &= (l_1 v_1 - l_0 v_0) + (l_1 \dot l_1 - l_0 \dot l_0) \\
    &= ((l_0 - k(t_1 - t_0))(v_0 +2k) - l_0 v_0) - k((l_0 - k(t_1 - t_0)) - l_0) \\
    &= 2k l_0 - k v_0 (t_1 - t_0) - k^2 (t_1 - t_0) \\
    &= k ( 2 l_0 - k (t_1 - t_0) - v_0 (t_1 - t_0) )\\
    &= 0,
\end{align*}
where we have used the fact that 
\begin{equation*}
    \begin{cases}
        l_1 = l_0 + \dot l_0 (t_1 - t_0) \\
        v_1 = v_0 - 2 \dot l_0
    \end{cases}
\end{equation*}

Next in the time direction 
\begin{equation*}
    2\mathcal{T}^{-1} (\theta_1 - \theta_0) = \int_{t_0}^{t_1} \dfrac{ds}{l(s)^2} = \dfrac{t_1-t_0}{l_0 l_1} \; .
\end{equation*}
And meanwhile
\begin{align*}
    \dfrac{l_0 l_1}{t_1-t_0} 
    &= \dfrac{l_0 v_0 l_1}{l_0 + l_1} \\
    &= \dfrac{l_0 v_0 (l_0 + l_1 - l_0)}{l_0 + l_1} \\
    &= l_0 v_0 - \dfrac{l_0 (l_1 -  \dot l_0 (t_1 - t_0))}{t_1 - t_0} \\
    &= \underbrace{l_0 v_0 + l_0 \dot l_0}_{I_0/\mathcal{T}} - \dfrac{l_0 l_1}{t_1 - t_0} \\
    &= \dfrac{I_0}{\mathcal{T}} - \dfrac{l_0 l_1}{t_1 - t_0},
\end{align*}
which implies that 
\[
\theta_1 - \theta_0 = \frac{1}{I_0} \; .
\]
\end{proof}

\subsubsection{The Adiabatic Normal Forms}

Here, we present the adiabatic normal forms (c.f. Proposition \ref{prop:adianorm}), defined on the Poincar\'e sections $R_0,R_T$, which capture the dramatic momentum change of the particle in roughly one period $\Delta t=2T$, due to the the singularities in the derivative of wall motions.\\

Starting with a high initial velocity, the particle experiences numerous collisions in a period. However, the adiabatic coordinates from Lemma \ref{lemma:adiacoord} makes it possible to ``sum up" these collisions until the particle foresees a singularity ahead of the next collision. We hereby perform another change of coordinates near the singularities and then complete the puzzle in the next proposition.

We define the ``new time" variable $\tau$ as follows: 
\[ \tau \buildrel \pmod 1 \over{:=} 
   \begin{cases}
       I\theta &\hbox{in $R_0$}\\
       I(\theta - \theta_T) &\hbox{in $R_T$}
   \end{cases}
\]
where $\theta_T=\displaystyle \dfrac{1}{2\mathcal{T}}\int_0^T \dfrac{ds}{l(s)^2} = \dfrac{1}{2}$.

\begin{prop}[Adiabatic normal forms]\label{prop:adianorm}
For all initial conditions $(\tau,I)$ with sufficiently large initial momentum $I$, the Poincar\'e maps of the particle are given by the following two adiabatic normal forms $P_1: R_0 \to R_T$ and $P_2: R_T \to R_0$ 
\begin{equation}\label{eq:adianorm1}
    P_1:=\begin{cases}
    \bar\tau=\tau-\dfrac{I}{2} \; \pmod 1\\
    \bar I=I+2Ak\mathcal{T}(2\bar\tau-1)
\end{cases},
\end{equation}

\begin{equation}\label{eq:adianorm2}
    P_2:=\begin{cases}
    \bar\tau=\tau-\dfrac{I}{2}\; \pmod 1\\
    \bar I=I-2Bk\mathcal{T}(2\bar\tau-1)
\end{cases}.
\end{equation}
\end{prop}

\begin{proof}
We present the proof for the Poincar\'e map $P_1$; the proof for $P_2$ can be done in a similar fashion. 

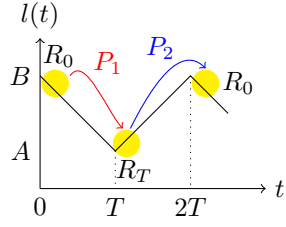
\begin{figure}[!ht]
    \centering
    \begin{tikzpicture}
        \filldraw[yellow] (0.2,1.4) circle (5pt);
        \filldraw[yellow] (1.15,0.6) circle (5pt);
        \filldraw[yellow] (2.2,1.4) circle (5pt);
        \draw[->] (0,0)--(3,0) node[anchor=west]{$t$};
        \draw[->] (0,0)--(0,2) node[anchor=south]{$l(t)$};
        \draw (0,1.5)--(1,0.5)--(2,1.5)--(2.5,1);
        \draw[dotted] (1,0.5)--(1,0) (2,1.5)--(2,0);
        \draw[->,red] (0.4,1.5)..controls (0.6,1.8) and (0.9,1)..(1.1,0.8);
        \draw[->,blue] (1.2,0.8)..controls (1.7,1.8) and (2,1.8)..(2.2,1.6);
        \draw[red] (0.9,1.4) node [anchor=south]{$P_1$};
        \draw[blue] (1.6,1.6) node [anchor=south]{$P_2$};
        \draw (0,1.5) node[anchor=east]{$B$}
              (0,0.5) node[anchor=east]{$A$}
              (0,0) node[anchor=north]{0}
              (1,0) node[anchor=north]{$T$}
              (2,0) node[anchor=north]{$2T$}
              (0.25,1.5) node[anchor=south]{$R_0$}
              (1.25,0.5) node[anchor=north]{$R_T$}
              (2.3,1.4) node[anchor=west]{$R_0$};
    \end{tikzpicture}
    \caption{The Poincar\'e sections and the adiabatic normal forms}
    \label{fig:adianorm}
\end{figure}

Let $(\theta_0,I_0)\in R_0$. We denote for ease $\tau=\theta_0 I_0, I=I_0$ and $(\bar{\tau},\bar{I}) = P_1 (\tau,I)$. 

By Lemma \ref{lemma:adiacoord}, $(\theta_n,I_n)\in f^{-1} R_0$ with $n=[I_0 (\theta_T - \theta_0)]$. Then $\bar{\tau} = I_{n+1}(\theta_{n+1}-\theta_T), \bar{I}=I_{n+1}$. 

Our goal is to derive formulas of $\bar{\tau},\bar{I}$ in terms of $\tau,I$. 

For $t_{n+1}\in(T,2T)$, $l(t_{n+1})=A+k(t_{n+1}-T), v_{n+1}=v_n -2k$, so we have 
\begin{equation}\label{eq:bartauI}
    \begin{cases}
        \bar{\tau} = \dfrac{I_{n+1}}{2\mathcal{T}} \displaystyle\int_0^1 \dfrac{ds}{l(s)^2} = \dfrac{I_{n+1}}{2\mathcal{T}} \dfrac{t_{n+1}-T}{A l_{n+1}}\\
        \bar{I} = \mathcal{T}(l_{n+1} v_{n+1} + l_{n+1} \dot l_{n+1}) = \mathcal{T} l_{n+1} (v_{n+1} +k) = \mathcal{T} l_{n+1} (v_n -k)
    \end{cases},
\end{equation}
which implies that 
\begin{equation}\label{eq:taubar1}
     \bar{\tau} = \dfrac{\mathcal{T} l_{n+1} (v_n +k)}{2\mathcal{T}} \dfrac{t_{n+1}-T}{A l_{n+1}} = \dfrac{1}{2A} (v_n -k) (t_{n+1}-T),
\end{equation}
which then suggests that we arrange computation in terms of $(v_n - k)$ and $(t_{n+1}-T)$.\\ 

We start with the formula for the momentum change. 

We recall that 
\begin{equation}\label{eq:In}
    I_n = \mathcal{T} (l_n v_n + l_n \dot l_n) = \mathcal{T} l_n (v_n -k).
\end{equation}
Then 
\begin{align*}
    I_{n+1} 
    &= \mathcal{T} l_{n+1} (v_n -k) \\
    &= \mathcal{T} (l_{n+1} -A + A - l_n + l_n ) (v_n -k) \\
    &= \mathcal{T} (l_{n+1} -A) (v_n -k) + \mathcal{T} (A - l_n ) (v_n -k) + \underbrace{\mathcal{T} l_n  (v_n -k)}_{I_n \hbox{ by \eqref{eq:In}}} \\
    &=  I_n + \mathcal{T} k \underbrace{(t_{n+1} -T) (v_n -k)}_{2A\bar{\tau} \hbox{ by \eqref{eq:taubar1}}} + \mathcal{T} (A - l_n ) (v_n -k) \\
    &= I_n + 2Ak\mathcal{T}\bar{\tau} - k\mathcal{T}(T-t_n)(v_n -k) .
\end{align*}
Now let us simplify $(T - t_n ) (v_n -k)$. 
\begin{align}
    v_n(t_{n+1}-t_n) &= l_n + l_{n+1} \nonumber\\
    \implies v_n (t_{n+1}-T + T - t_n) &= (l_{n+1} - A) - (A-l_n) + 2A \nonumber\\
    \implies \underbrace{(v_n - k)(t_{n+1}-T)}_{2A\bar{\tau} \hbox{ by \eqref{eq:taubar1}}} &= (k-v_n)(T - t_n) + 2A \nonumber\\
    \implies (v_n -k)(T - t_n) &= -2A(\bar{\tau} -1). \label{eq:vntnT}
\end{align}
Hence we continue 
\begin{align*}
    I_{n+1} 
    &= I_n + 2Ak\mathcal{T}\bar{\tau} + 2Ak\mathcal{T}(\bar\tau -1)\\
    &= I_n + 2Ak\mathcal{T} (2\bar\tau -1) 
\end{align*}
which implies, together with $I_n = I_0$ by Lemma \ref{lemma:adiacoord}, that 
\[
\bar{I} = I + 2Ak\mathcal{T} (2\bar\tau -1) \; .
\]

\vspace{5pt}

Next we derive the formula for the time change. 

We note that in the $n^{th}$ collision right before the singularity at $t=T$ 
\begin{align*}
    I_n \theta_n 
    &= I_n (\theta_T - (\theta_T - \theta_n)) \\
    &= I_n \theta_T - \mathcal{T} (l_n v_n - k l_n) \dfrac{1}{2\mathcal{T}} \int_{t_n}^T \dfrac{ds}{l(s)^2} \\
    &= \dfrac{1}{2} I_0 - \dfrac{1}{2} l_n (v_n -k)\dfrac{T-t_n}{A l_n} \\
    &= \dfrac{1}{2} I_0 - \dfrac{1}{2A} (v_n -k) (T-t_n).
\end{align*}
Then by \eqref{eq:vntnT} we have 
\begin{align*}
    \bar\tau 
    &= 1- \dfrac{1}{2A} (v_n -k) (T-t_n) \\
    &= 1 + I_n\theta_n - \dfrac{1}{2} I_0 \\
    &= 1 + I_0 \left(\theta_0 + \dfrac{n}{I_0}\right) - \dfrac{1}{2} I_0 \\
    &= I_0 \theta_0 - \dfrac{1}{2} I_0 + n+1 \\
    &\buildrel \pmod{1} \over = \tau - \dfrac{1}{2} I_0 \; .
\end{align*}
\end{proof}

\begin{rmk}\label{rmk:onlyhighenergy}
    We emphasize that the adiabatic coordinates and normal forms are only valid for large energies. In fact, the first equation in \eqref{eq:collisoneqs} is not true if $v_0$ is so small that a \emph{re-collision} occurs, i.e. the moving wall chases and collides the particle before the particle collides at the fixed wall.
\end{rmk}

\subsection{Skew Product Structure under Resonance}\label{ssec:skew}

In this subsection we show that the complete Poincar\'e map $P:=P_2 \circ P_1$ possesses invariant circles. If furthermore the parameters $A,B$ satisfy resonance conditions, then the parabolic dynamics on the invariant circles can be decomposed as a skew product over an interval exchange map with integer change on the fibers. 

\subsubsection{Invariant Circles}
Since $P_1,P_2$ from Proposition \ref{prop:adianorm} are (piecewise) affine maps with constant coefficients, the eigendirection of the complete Poincar\'e map $P=P_2 \circ P_1$ is always preserved. Now we verify that the translations in the affine maps are compatible with the dynamics so that we obtain invariant circles of $P$ on the $(\tau,I)$-phase cylinder.\\

We note that 
\[ DP_1 = 
\begin{pmatrix}
    1 & -\dfrac{1}{2} \\
    4Ak\mathcal{T} & 1-2Ak\mathcal{T}
\end{pmatrix}, \;
DP_2 = \begin{pmatrix}
    1 & -\dfrac{1}{2} \\
    -4Bk\mathcal{T} & 1+2Bk\mathcal{T}
\end{pmatrix}
\]
and consequently 
\[ DP = DP_2 \cdot DP_1 = 
\begin{pmatrix}
    1-2Ak\mathcal{T} & Ak\mathcal{T}-1\\
    4(B-A)k\mathcal{T} & 1+2Ak\mathcal{T}
\end{pmatrix}
\]
where we have used the fact that $\mathcal{T}=\dfrac{T}{AB}$ and $k=\dfrac{B-A}{T}$.

We further note that $\det(DP)=1$, and $DP$ has an eigenvalue $\lambda=1$ with an eigenvector $\mathbf{v}=(A, 2(A-B))$. We propose the following candidate as the family $\{\mathcal{C}_D\}_{D\in [0,1)}$ of invariant circles of the complete Poincar\'e map $P$:
\begin{equation}\label{eq:invcirc}
    \mathcal{C}_D := \left\{(\tau,I)\, \big| \, \tau + \dfrac{A}{2(B-A)} = D \pmod1 \right\}.
\end{equation}

We claim that $\mathcal{C}_D$ is invariant under the dynamics $P$.

\begin{lemma}[Invariant circles]\label{lemma:invcirc}
    For any $D\in[0,1)$, $\mathcal{C}_D$ as defined in \eqref{eq:invcirc} is invariant by the map $P$.
\end{lemma}

\begin{proof}
    Fix $D\in(0,1)$. Suppose $(\tau,I)\in\mathcal{C}_D$, i.e. $\tau + \dfrac{A}{2(B-A)} I = D \pmod1$. We denote $(\bar{\tau},\bar{I}) = P_1 (\tau,I)$ and $(\bar{\bar{\tau}},\bar{\bar{I}}) = P_2 (\bar\tau,\bar{I})$. 

    By Proposition \ref{prop:adianorm}, we have 
    \begin{align*}
        \bar\tau &= \tau - \dfrac{I}{2} \mod1\\
        \bar I &= I + 2Ak\mathcal{T}\left(2\left(\tau - \dfrac{I}{2} \mod1 \right) -1 \right)\\
        &= I + 4Ak\mathcal{T}\left(\tau - \dfrac{I}{2} \mod1 \right) - 2Ak\mathcal{T}
    \end{align*}
    and then 
    \begin{align*}
        \bar{\bar{\tau}} &= \bar\tau - \dfrac{\bar I}{2} \mod1\\
        &= \left(\tau - \dfrac{I}{2} \mod1 \right) - \dfrac{I}{2} - 2Ak\mathcal{T} \left(\tau - \dfrac{I}{2} \mod1 \right) + Ak\mathcal{T} \mod1\\
        &= \tau - I - 2Ak\mathcal{T} \left(\tau - \dfrac{I}{2} \mod1 \right) + Ak\mathcal{T} \mod1 \\
        \bar{\bar{I}} &= \bar{I} - 2Bk\mathcal{T} (2\bar{\bar{\tau}} -1)\\
        &= I + 4Ak\mathcal{T} \left(\tau - \dfrac{I}{2} \mod1 \right) - 2Ak\mathcal{T} + 2Bk\mathcal{T} \\
        &\quad - 4Bk\mathcal{T}\left(\tau -I - 2Ak\mathcal{T} \left(\tau - \dfrac{I}{2} \mod1 \right) + Ak\mathcal{T} \mod1 \right) \; .
    \end{align*}

    We verify that $(\bar{\bar{\tau}},\bar{\bar{I}}) \in \mathcal{C}_D$.
    \begin{align*}
        \bar{\bar{\tau}} + \dfrac{A}{2(B-A)} \bar{\bar{I}} 
        &=_1 \textcolor{red}{\tau} - I - 2Ak\mathcal{T} \left(\tau - \dfrac{I}{2} \mod1 \right) + Ak\mathcal{T} + \textcolor{red}{\dfrac{A}{2(B-A)} I} \\
        &\quad + \underbrace{\dfrac{2A^2 k\mathcal{T}}{B-A}}_{2A/B} \left(\tau - \dfrac{I}{2} \mod1 \right) - \underbrace{\dfrac{A^2 k\mathcal{T}}{B-A}}_{A/B} + \underbrace{\dfrac{AB k\mathcal{T}}{B-A}}_{1} \\
        &\quad -\underbrace{\dfrac{2AB k\mathcal{T}}{B-A}}_{2} \left(\tau -I - 2Ak\mathcal{T} \left(\tau - \dfrac{I}{2} \mod1 \right) + Ak\mathcal{T} \mod1 \right)\\
        &=_1 D + \left(\tau - \dfrac{I}{2} \mod1 \right)\underbrace{\left(-2Ak\mathcal{T} + \dfrac{2A}{B} + 4Ak\mathcal{T}\right)}_{2} \\
        &\quad + I -2\tau + \underbrace{Ak\mathcal{T} - \dfrac{A}{B} - 2Ak\mathcal{T} + 1}_{0} \\
        &=_1 D + 2\tau - I + I - 2\tau \\
        &=_1 D \; .
    \end{align*}
\end{proof}

\subsubsection{Skew Product Structure}
Now, we show that under resonance condition, the invariant circle $\mathcal{C}_D$ admits a skew product structure and the dynamics on $\mathcal{C}_D$ is equivalent to a finite extension of circle rotations with a piecewise constant cocycle.

\begin{definition}[Classical resonance]\label{def:resonance}
    For $q\in\mathbb{N}$, we say the parameters $(A,B)$ are \emph{in (classical) $q$-resonance} if $(B-A)/A=q$.
\end{definition}

By our assumption $B>A$, $q\ge 1$. We say $(A,B)$ is \emph{in special resonance} if $(B-A)/A=1$. The particular choice of parameters $(A,B,T)=(1/\sqrt{2}, \sqrt{2}, 1)$ made by Ulam \cite{Ulam61} and Zharnitsky \cite{Zha98} corresponds to a special case of the special resonance.\\

We introduce some necessary notations before we state the main proposition of the section. 

\begin{definition}[A/typical circles]\label{def:Dtyp}
    We say an invariant circle $\mathcal{C}_D$ is \emph{typical} if $D \ne m/q$ for any integer $m=0,\cdots,q-1$. Otherwise, we say $\mathcal{C}_D$ is \emph{atypical} if $D = m/q$ for some integer $m=0,\cdots,q-1$.
\end{definition}

We cut an invariant circle $\mathcal{C}_{D}$ into continuity components of $P_1, P_2$ (c.f. Fig.\ref{fig:contcomp}).

Fix an integer $m=0,\cdots,q-1$. 

For a typical circle $\mathcal{C}_D$ with $D\in\left( \dfrac{m}{q}, \dfrac{m+1}{q} \right)$, we define for any $s=0,1,\cdots,q+1$ 
\begin{equation}\label{eq:typbasecomp}
    \mathcal{C}_{D}^{m,s} := \mathcal{C}_{D} \cap \left\{ \dfrac{q}{q+1}D - \dfrac{1}{q+1}(m+1-s) < \tau < \dfrac{q}{q+1}D - \dfrac{1}{q+1}(m-s) \right\}.
\end{equation}

We need to introduce the secondary cut on the starting or ending component, depending on whether $D$ lies in the first or the second half of $\big( m/q, (m+1)/q \big)$. 

For $D\in\left( \dfrac{m}{q}, \dfrac{m+\frac{1}{2}}{q} \right]$, we perform secondary cut on the ending component $\mathcal{C}_{D}^{m,q+1}$ at $\tau^{m,q+1}:= 2-\frac{q}{q+1}(1-2D)-\frac{2}{q+1}(m+1)$
\begin{equation}\label{eq:seccutq}
    \mathcal{C}_{D}^{m,q+1-} := \mathcal{C}_{D}^{m,q+1} \cap \{\tau < \tau^{m,q+1}\}, \; \mathcal{C}_{D}^{m,q+1+} := \mathcal{C}_{D}^{m,q+1} \cap \{\tau > \tau^{m,q+1}\}.
\end{equation}

For $D\in\left( \dfrac{m+\frac{1}{2}}{q} , \dfrac{m+1}{q} \right)$, we perform secondary cut on the starting component $\mathcal{C}_{D}^{m,0}$ at $\tau^{m,0}:= 1-\frac{q}{q+1}(1-2D)-\frac{2}{q+1}(m+1)$
\begin{equation}\label{eq:seccut0}
    \mathcal{C}_{D}^{m,0-} := \mathcal{C}_{D}^{m,0} \cap \{\tau < \tau^{m,0}\}, \; \mathcal{C}_{D}^{m,0+} := \mathcal{C}_{D}^{m,0} \cap \{\tau > \tau^{m,0}\}.
\end{equation}

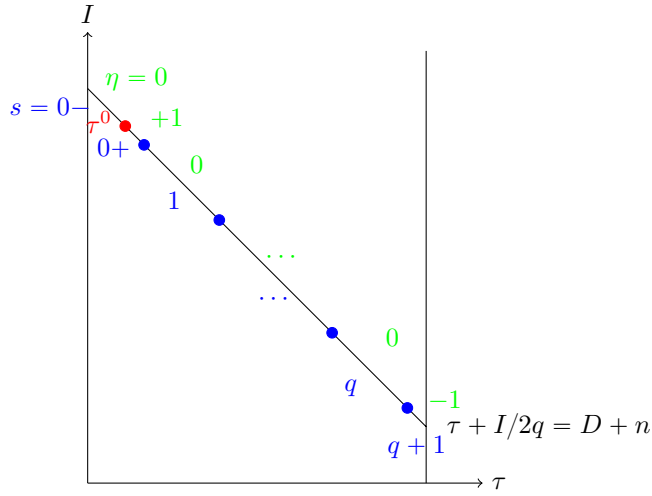
\begin{figure}[!ht]
    \centering
    \begin{tikzpicture}
        \draw[->] (-0.75,-1)--(4.5,-1) node[anchor=west]{$\tau$};
        \draw[->] (-0.75,-1)--(-0.75,5) node[anchor=south]{$I$};
        \draw (3.75,-1)--(3.75,4.75) (-0.75,4.25)--(3.75,-0.25);
        \foreach \x/\y in {0/3.5, 1/2.5, 2.5/1, 3.5/0} {
        \filldraw[fill=blue, draw=blue] (\x,\y) circle (2pt);
        }
        \filldraw[fill=red, draw=red] (-0.25,3.75) circle (2pt);
        \draw[red] (-0.3,3.8) node[anchor=east]{$\tau^0$};
        \foreach \x/\y/\text in {-1.25/4.25/$s=0-$, -0.4/3.7/$0+$, 0.4/3/1, 1.75/1.65/$\cdots$, 2.75/0.5/$q$, 3.625/-0.25/$q+1$} {
        \node[blue,anchor=north] at (\x,\y) {\text};
        }
        \foreach \x/\y/\text in {-0.1/4.1/$\eta=0$, 0.3/3.6/$+1$, 0.7/3/0, 1.85/1.8/$\cdots$, 3.3/0.7/0, 4/-0.15/$-1$} {
        \node[green,anchor=south] at (\x,\y) {\text};
        }
        \draw (3.9,-0.25) node[anchor=west]{$\tau + I/2q = D+n$};
    \end{tikzpicture}
    \caption{Continuity components $\mathcal{C}_{D,n}^{m,s}$, $D>(m+\frac{1}{2})/q$}
    \label{fig:contcomp}
\end{figure}

For an atypical circle $\mathcal{C}_{m/q}$ ($m=0,1,\cdots, q-1$), the 0$^{th}$ component $\mathcal{C}_{m/q}^{0}$ vanishes and we have only the remaining $q+1$ components $\mathcal{C}_{m/q}^{s}$ with $s=1,\cdots, q+1$. No secondary cut is required for atypical circles.

\begin{prop}[Skew product structure]\label{prop:skewprod}
    Fix $q\in\mathbb{N}$. Assume that the parameters $(A,B)$ are in $q$-resonance as in Definition \ref{def:resonance}. Then the map $P$ restricted on $\mathcal{C}_D$ is equivalent to the skew product $\eta_F (\tau,n) := (F(\tau),n+\eta(\tau))$ on $\mathbb{S}^1 \times \mathbb{Z}$, where $F:\mathbb{S}^1 \circlearrowleft$ is an interval exchange map, and $\eta:\mathbb{S}^1\to\{0,1,-1\}$ is piecewise constant. 
   
   More precisely, suppose that $D\in\left[ \dfrac{m}{q}, \dfrac{m+1}{q} \right)$ for some $m=0,\cdots,q-1$. Then on the base $\mathbb{S}^1$, for any $s=0,\cdots,q+1$, on $\mathcal{C}_{D}^{m,s}$ we have 
    \begin{equation}\label{eq:iet}
        F(\tau) = \tau + \dfrac{q}{q+1}(1-2D) + \dfrac{2}{q+1}(m+1-s) \pmod1 .
    \end{equation}
    
    On the fiber $\mathbb{Z}$, if $D\in\left( \dfrac{m}{q}, \dfrac{m+\frac{1}{2}}{q} \right]$, then 
    \begin{equation}\label{skewprodtyp1}
        \eta (\tau) = 
        \begin{cases}
            0 & \mathcal{C}_D^{m,s}, \; s=1,\cdots,q,q+1+\\
            1 & \mathcal{C}_D^{m,0}\\
            -1 & \mathcal{C}_D^{m,q+1-}
        \end{cases};
    \end{equation}

    if $D\in\left( \dfrac{m+\frac{1}{2}}{q} , \dfrac{m+1}{q} \right)$, then  
    \begin{equation}\label{skewprodtyp2}
        \eta (\tau) = 
        \begin{cases}
            0 & \mathcal{C}_D^{m,s}, \; s=1,\cdots,q,0-\\
            1 & \mathcal{C}_D^{m,0+}\\
            -1 & \mathcal{C}_D^{m,q+1}
        \end{cases};
    \end{equation}

    and if $D=m/q$, then  
    \begin{equation}\label{skewprodatyp}
        \eta (\tau) \equiv 0.
    \end{equation}
\end{prop}

\begin{proof}
We begin with typical $D\in\left( \dfrac{m}{q}, \dfrac{m+1}{q} \right)$.

For $n\in\mathbb{Z}$, we define 
\begin{equation}\label{eq:Cdn}
    \mathcal{C}_{D,n} := \left\{(\tau,I)\, \big| \, \tau + \dfrac{I}{2q} = D+n \right\}.
\end{equation}  
The invariant circle $\mathcal{C}_{D}$ is decomposed into the family $\{\mathcal{C}_{D,n}\}_n$ of short lines indexed by the ``floor number" $n$. 

On each floor $\mathcal{C}_{D,n}$, $\tau - \dfrac{I}{2}$ ranges within $\big(-qn -(m+1), -qn-(m-q) \big)$, since $D\in \left( \dfrac{m}{q}, \dfrac{m+1}{q} \right)$. The singularities arising from $\tau - I/2 \in\mathbb{Z}$ cut $\mathcal{C}_{D,n}$ into $q+1$ components $\mathcal{C}_{D,n}^{m,s}$, $s=0,\cdots,q+1$ (c.f. the blue part in Fig.\ref{fig:contcomp}). The cutting is determined by the following equations 
\begin{equation}\label{eq:cut1}
    \begin{cases}
        \tau + \dfrac{I}{2q} = D+n \\
        \tau - \dfrac{I}{2}  = -qn - (m+1-s) \; .
    \end{cases} 
\end{equation}

We denote $(\bar{\tau},\bar{I}) = P_1(\tau,I)$. On each $\mathcal{C}_{D,n}^{m,s}$ ($s=0,\cdots,q+1$), by Proposition \ref{prop:adianorm} we have 
\begin{equation}\label{eq:taubaribar1}
    \begin{aligned}
        \bar\tau 
    &= \tau - \dfrac{I}{2} \pmod1\\
    &= \tau - \dfrac{I}{2} + (m+1-s) + qn \; \in (0,1) \; ,\\
    \bar{I} 
    &= I + \dfrac{2q}{q+1} (2\bar\tau - 1)\\
    &= I + \dfrac{2q}{q+1} \big( 2\tau - I + 2(m-s) + 2qn + 1 \big) \; .
    \end{aligned}
\end{equation}

\vspace{5pt}

Next, the singularities arising from $\bar{\tau}-\dfrac{\bar{I}}{2} \in \mathbb{Z}$ may require further cut on $\mathcal{C}_{D,n}^{m,s}$. We denote $(\bar{\bar{\tau}},\bar{\bar{I}}) = P_2(\bar\tau,\bar I) = P_2 \circ P_1 (\tau,I)$.

For $s=0,\cdots,q+1$, by \eqref{eq:taubaribar1} we have 
\begin{equation*}
    \begin{aligned}
        \bar{\tau}-\dfrac{\bar{I}}{2} 
        &= \dfrac{1-q}{q+1} \big( \tau - \dfrac{I}{2} + (m+1-s) + qn \big) - \dfrac{I}{2} + \dfrac{q}{q+1} \\
        &\buildrel \eqref{eq:Cdn} \over = \tau + \dfrac{q}{q+1}(1-2D) + \dfrac{2}{q+1}(m+1-s) - (m+1-s) - qn \; .
    \end{aligned} 
\end{equation*}

For the middle components $\mathcal{C}_{D,n}^{m,s}$, $s=1,\cdots,q$, by solving \eqref{eq:cut1} the left endpoint $L_s$ and right endpoint $R_s$ of $\mathcal{C}_{D,n}^{m,s}$ are respectively 
\begin{equation*}
    L_s :
    \begin{cases}
        \tau_L^s = \dfrac{q}{q+1}D - \dfrac{1}{q+1}(m+1-s) \\
        I_L^s = 2q \left(D+n - \dfrac{q}{q+1}D + \dfrac{1}{q+1}(m+1-s)\right) \; ,
    \end{cases}
\end{equation*}
and 
\begin{equation*}
    R_s :
    \begin{cases}
        \tau_R^s = \dfrac{q}{q+1}D - \dfrac{1}{q+1}(m-s) \\
        I_R^s = 2q \left(D+n - \dfrac{q}{q+1}D + \dfrac{1}{q+1}(m-s) \right).
    \end{cases}
\end{equation*}

The values of $\bar{\tau}-\dfrac{\bar{I}}{2}$ at the endpoints $L_s, R_s$ are respectively 
\begin{align*}
    \bar{\tau}-\dfrac{\bar{I}}{2} \bigg\vert_{L_s} 
    &= -\dfrac{q}{q+1}D + \dfrac{1}{q+1}(m+1-s) - (m+1-s) - qn + \dfrac{q}{q+1} \; , \\
    \bar{\tau}-\dfrac{\bar{I}}{2} \bigg\vert_{R_s} 
    &= -\dfrac{q}{q+1}D + \dfrac{1}{q+1}(m-s) - (m+1-s) - qn + \dfrac{q}{q+1} \; .
\end{align*}
Since $D\in\left( \dfrac{m}{q}, \dfrac{m+1}{q} \right)$, $\bar{\tau}-\dfrac{\bar{I}}{2} \bigg\vert_{L_s}$ and $\bar{\tau}-\dfrac{\bar{I}}{2} \bigg\vert_{R_s}$ both take value in the same integer interval 
\[
\big( -(m+1-s)-qn,-(m-s)-qn \big).
\]
Consequently no further cutting is required and 
\begin{align}
    \bar{\bar{\tau}} &= \bar{\tau}-\dfrac{\bar{I}}{2} \pmod1 \nonumber\\
    &= \bar{\tau}-\dfrac{\bar{I}}{2} + (m+1-s) + qn \; \in (0,1) \nonumber\\
    &= \tau + \dfrac{q}{q+1}(1-2D) + \dfrac{2}{q+1}(m+1-s) \; , \label{eq:barbartau}
\end{align}
which agrees with \eqref{eq:iet}.

\vspace{5pt}

Now we compute the floor change $\eta$ for the middle components. We need to compute the new floor number $\bar{\bar{n}}$ for the image point $(\bar{\bar{\tau}},\bar{\bar{I}}) \in \mathcal{C}_{D,\bar{\bar{n}}}$. 
\begin{align*}
    \bar{\bar{\tau}} + \dfrac{\bar{\bar{I}}}{2q}
    &= \bar{\bar{\tau}} + \dfrac{\bar{I}}{2q} - (2\bar{\bar{\tau}} -1) = - \bar{\bar{\tau}} + \dfrac{\bar{I}}{2q} +1\\
    &\buildrel \eqref{eq:barbartau} \over = - \tau - \dfrac{q}{q+1}(1-2D) - \dfrac{2}{q+1}(m+1-s) + \dfrac{I}{2q} + \dfrac{2\bar\tau -1}{q+1} + 1 \\
    &= \underbrace{\tau + \dfrac{I}{2q}}_{D+n} - 2\tau - \dfrac{q}{q+1}(1-2D) - \dfrac{2}{q+1}(m+1-s) + \dfrac{2\tau -I}{q+1} \\
    &\quad +\dfrac{2(m+1-s)+2qn-1}{q+1} +1\\
    &= D + n - \dfrac{2q}{q+1}\underbrace{\left(\tau + \dfrac{I}{2q}\right)}_{D+n} - \dfrac{q}{q+1}(1-2D) + \dfrac{2qn-1}{q+1} +1\\
    &= D + n + D\left(-\dfrac{2q}{q+1}+\dfrac{2q}{q+1}\right) + n \left(-\dfrac{2q}{q+1}+\dfrac{2q}{q+1}\right) \\
    &\quad - \dfrac{q}{q+1} - \dfrac{1}{q+1} +1 \\
    &= D + n \; .
\end{align*}
Therefore $\bar{\bar{n}} = n$ (or equivalently $\eta=0$) on the middle components $\mathcal{C}_{D,n}^{m,s}$, $s=1,\cdots,q$, which verifies the corresponding parts in \eqref{skewprodtyp1} \eqref{skewprodtyp2}.

\vspace{10pt}

Next, on the starting component $\mathcal{C}_{D,n}^{m,0}$, we need a further cut only when $D\in\left(\dfrac{m+\frac{1}{2}}{q},\dfrac{m+1}{q}\right)$.

Indeed, the formulas for right endpoint $R_s$ still hold for $s=0$, but the left endpoint point is now at the boundary $L_0=0$ of the phase cylinder, and hence 
\begin{equation*}
    \bar{\tau}-\dfrac{\bar{I}}{2} \bigg\vert_{L_0} 
    = -\dfrac{2q}{q+1}D + \dfrac{2m+1}{q+1} - m - qn \; .
\end{equation*}

If $D\in\left(\dfrac{m}{q},\dfrac{m+\frac{1}{2}}{q}\right]$, then $\bar{\tau}-\dfrac{\bar{I}}{2} \bigg\vert_{L_0}$ and $\bar{\tau}-\dfrac{\bar{I}}{2} \bigg\vert_{R_0}$ both take value in the same integer interval 
\[
\big[ -m-qn,-m-qn + 1 \big)
\]
and consequently no further cutting is required. In this case, we follow the same procedures as for the middle components and obtain 
\begin{equation*}
    \begin{cases}
        \bar{\bar{\tau}} &= \tau + \dfrac{q}{q+1}(1-2D) + \dfrac{2}{q+1}(m+1) -1 \; \in (0,1) \\
        \bar{\bar{n}} &= n+1 \;
    \end{cases}
\end{equation*}
which verifies \eqref{eq:iet} and the corresponding part for the component $\mathcal{C}_{D,n}^{m,0}$ in \eqref{skewprodtyp1}.

\vspace{5pt}

If $D\in\left(\dfrac{m+\frac{1}{2}}{q},\dfrac{m+1}{q}\right)$, then $\bar{\tau}-\dfrac{\bar{I}}{2} \bigg\vert_{L_0}$ and $\bar{\tau}-\dfrac{\bar{I}}{2} \bigg\vert_{R_0}$ now take value in two different integer intervals
\begin{equation*}
   \begin{split}
       \bar{\tau}-\dfrac{\bar{I}}{2} \bigg\vert_{L_0} \in \big( -m-qn -1 ,-m-qn \big) \\
       \bar{\tau}-\dfrac{\bar{I}}{2} \bigg\vert_{R_0} \in \big( -m-qn,-m-qn + 1 \big) 
   \end{split}
\end{equation*}
and consequently a further cutting is required at $\tau^{0} = 1-\frac{q}{q+1}(1-2D)-\frac{2}{q+1}(m+1)$ (c.f. the red part in Fig.\ref{fig:contcomp}), which is solved from $\bar{\tau}-\dfrac{\bar{I}}{2}=-m-qn$.

By following the same procedures as for the middle components, we obtain on the component $\mathcal{C}_{D,n}^{m,0-}$
\begin{equation*}
    \begin{cases}
        \bar{\bar{\tau}} &= \tau + \dfrac{q}{q+1}(1-2D) + \dfrac{2}{q+1}(m+1) \; \in (0,1) \\
        \bar{\bar{n}} &= n \;
    \end{cases}
\end{equation*}
and on the component $\mathcal{C}_{D,n}^{m,0+}$
\begin{equation*}
    \begin{cases}
        \bar{\bar{\tau}} &= \tau + \dfrac{q}{q+1}(1-2D) + \dfrac{2}{q+1}(m+1) -1 \; \in (0,1) \\
        \bar{\bar{n}} &= n+1 \;
    \end{cases}
\end{equation*}
which verifies \eqref{eq:iet} and the corresponding part for the components $\mathcal{C}_{D,n}^{m,0\pm}$ in \eqref{skewprodtyp2}.

Similarly, for the ending component $\mathcal{C}_{D,n}^{m,q+1}$, we need a further cut only when $D\in\left(\dfrac{m}{q},\dfrac{m+\frac{1}{2}}{q}\right]$, and we cut at $\tau^{q+1}= 2-\frac{q}{q+1}(1-2D)-\frac{2}{q+1}(m+1)$. By similar computations, we obtain that if $D\in\left(\dfrac{m}{q},\dfrac{m+\frac{1}{2}}{q}\right]$, then on the component $\mathcal{C}_{D,n}^{m,q+1-}$
\begin{equation*}
    \begin{cases}
        \bar{\bar{\tau}} &= \tau + \dfrac{q}{q+1}(1-2D) + \dfrac{2}{q+1}(m+1-(q+1)) +1 \; \in (0,1) \\
        \bar{\bar{n}} &= n -1 \;
    \end{cases}
\end{equation*}
and on the component $\mathcal{C}_{D,n}^{m,q+1+}$
\begin{equation*}
    \begin{cases}
        \bar{\bar{\tau}} &= \tau + \dfrac{q}{q+1}(1-2D) + \dfrac{2}{q+1}(m+1-(q+1)) \; \in (0,1) \\
        \bar{\bar{n}} &= n \; ;
    \end{cases}
\end{equation*}
and if $D\in\left(\dfrac{m+\frac{1}{2}}{q},\dfrac{m+1}{q}\right)$, then on the component $\mathcal{C}_{D,n}^{m,q+1}$
\begin{equation*}
    \begin{cases}
        \bar{\bar{\tau}} &= \tau + \dfrac{q}{q+1}(1-2D) + \dfrac{2}{q+1}(m+1-(q+1)) +1 \; \in (0,1) \\
        \bar{\bar{n}} &= n-1 \; ,
    \end{cases}
\end{equation*}
which again verifies \eqref{eq:iet} and the remaining parts in \eqref{skewprodtyp1} \eqref{skewprodtyp2}.

\vspace{5pt}

Lastly, for atypical $D=\dfrac{m}{q}$, the leftmost and rightmost cut occur at exactly the left and boundaries of the phase cylinder, producing exactly $q+1$ components $\mathcal{C}_{D,n}^{m,s}$, $s=1,\cdots,q+1$, all of uniform length $1/(q+1)$. It is straightforward to see that they all behave the same way as the middle components in the typical case: no further cut is required and on each component $\mathcal{C}_{D,n}^{m,s}$, $s=1,\cdots,q+1$ 
\begin{equation*}
    \begin{cases}
        \bar{\bar{\tau}} &= \tau + \dfrac{q}{q+1}(1-2D) + \dfrac{2}{q+1}(m+1-s) +1\\
        &= \tau + 1 +\dfrac{1-2s}{q+1} \; \in (0,1) \\
        \bar{\bar{n}} &= n \;
    \end{cases}
\end{equation*}
which verifies \eqref{eq:iet} and \eqref{skewprodatyp}.
\end{proof}

\begin{rmk}\label{rmk:ietq+1}
    We observe from \eqref{eq:iet} that for typical $D$, $F$ has the same definition on the starting and ending components $\mathcal{C}_D^{m,0},\mathcal{C}_D^{m,q+1}$, and hence we may link these two components into one and then $F$ is now an interval exchange map of $q+1$ intervals, each of uniform length $1/(q+1)$. The same observation is automatically true for atypical $D$.
\end{rmk}

\subsection{Non-escaping and Recurrence}\label{ssec:rec}

In this section we prove the main result Theorem \ref{thm:rec}, i.e. recurrence of the classical piecewise linear Fermi-Ulam models under resonance, by showing that the skew product structure on an irrational invariant circle has an ergodic base and a zero-average cocycle.

\subsubsection{Ergodicity of the Irrational Base Map}

Firstly, we show that the base map $F$ on an irrational circle is ergodic.

\begin{prop}\label{prop:erg}
    When restricted on an invariant circle $\mathcal{C}_D$ with $D$ irrational, the map $F$ is ergodic with respect to the Lebesgue measure.
\end{prop}

Before the proof, let us simplify the map $F$. We fix $D\notin \mathbb{Q}$.

Recall from Remark \ref{rmk:ietq+1} that on an invariant circle $\mathcal{C}_D$ the base map $F$ is an interval exchange map of $q+1$ intervals, each of length $1/(q+1)$. In fact, $F$ is isomorphic to a finite extension of circle rotations. Indeed, by \eqref{eq:iet}, we see that, $$F\left(\tau+\frac{1}{q+1}\right)\pmod{\frac{1}{q+1}}=F(\tau)\pmod{1/(q+1)},$$therefore, $F$ factors onto ${\mathbb S}^{1/(q+1)}$, with the factor map $\tilde F$ given by $$\tilde{F}(\tau \pmod{1/(q+1)})=\tau+\frac{-2q}{q+1}D\pmod{1/(q+1)}.$$From above, it yields that $F$ is isomorphic to a skew product system defined on $$\Sigma:=\mathbb{S}^{1/(q+1)} \times \{1,\cdots,q+1\},$$ of the form $$H(x,k)=(\tilde{F}(x),\sigma_x(k)),\;\forall\, (x,k)\in\Sigma,$$where $\sigma_x$ is a permutation of $\{1,\cdots,q+1\}$. A simple computation shows that $$\sigma_x(k)\buildrel \mod{(q+1)} \over =\begin{cases}
    2m-k+[q(1-2D)],&\text{if} \;x\in [0,\tau^*)\\
    2m-k+[q(1-2D)]+1,&\text{if}\; x\in [\tau^*,\frac{1}{q+1})
\end{cases}$$ where $\tau^*=\frac{1-\{q(1-2D)\}}{q+1}\in (0,\frac{1}{q+1})$.

Now for simplicity, let $\varkappa=\{-2qD\}$. After converting $\mathbb{S}^{\frac{1}{q+1}}$ to $\mathbb{S}^1$, we find that $H$ is isomorphic to, with abuse of notation, $$\hat H(x,k)=(x+\varkappa,\sigma_x(k)),$$where $$\sigma_x(k)=\begin{cases}
    2m-k+[q(1-2D)],&\text{if} \;x\in [0,1-\varkappa);\\
    2m-k+[q(1-2D)]+1,&\text{if}\; x\in [1-\varkappa,1).
\end{cases}$$

Denote $\sigma^1(k):=2m-k+[q(1-2D)]$, and $\sigma^2(k):=2m-k+[q(1-2D)]+1$, for any $k\in\{1,\cdots, q+1\}$. Let $\sigma^3(k):=k+1$. A simple calculation shows that $\sigma^1\circ\sigma^1=\sigma^2\circ\sigma^2=I$, $\sigma^2\circ\sigma^1=\sigma^3$ and $\sigma^1\circ\sigma^2=(\sigma^3)^{-1}$. This observation motivates us to consider $\hat H\circ \hat H$ instead of $\hat H$ directly. 

We find that $\hat H\circ \hat H$ is isomorphic to $h(x,k)=(x+2\varkappa,k+\phi(x))$. Here if $\varkappa\in (0,\frac{1}{2})$, $$\phi(x)=\begin{cases}
    0,&\text{if} \;x\in [0,1-2\varkappa);\\
    1,&\text{if}\; x\in [1-2\varkappa,1-\varkappa);\\
    -1,&\text{if}\; x\in [1-\varkappa,1).
\end{cases}$$If otherwise $\varkappa\in(\frac{1}{2},1)$, $$\phi(x)=\begin{cases}
    1,&\text{if} \;x\in [0,1-\varkappa);\\
    -1,&\text{if}\; x\in [1-\varkappa,2-2\varkappa);\\
    0,&\text{if}\; x\in [2-2\varkappa,1).
\end{cases}$$

\begin{prop}\label{prop:erg1}
    The map $h$ defined on $\mathbb{S}^1\times \{1,\cdots,q+1\}$, is ergodic.
\end{prop}

\begin{proof}[Proof of Proposition \ref{prop:erg}]
    By Proposition \ref{prop:erg1} and the discussion above, the ergodicity of $h$ implies that $\hat H\circ\hat H$ is ergodic, and then $\hat H$, as well as $H$ are ergodic. As a result of the above reduction, $F$ is ergodic. 
\end{proof}

To prove Proposition \ref{prop:erg1}, we need the following key result from \cite{SP}, which gives a minimality criteria for finite group extensions of interval exchange transformations (including circle rotations). Fix $q\in\mathbb{N}$ and we consider 
\begin{equation*}
\begin{split}
    &T_\phi:\mathbb S^1\times \mathbb Z/q\mathbb Z\to \mathbb S^1\times \mathbb Z/q\mathbb Z, \\
     &T_\phi(x,k)=(Tx,k+\phi(x)) \; .
\end{split}   
\end{equation*}

\begin{lemma}[\cite{SP}, Theorem 10]\label{lemma:criteria}
  Let $T$ be a minimal aperiodic $r$-interval exchange transformation with discontinuities $\gamma_i$, $1\le i\le N$. Let
\begin{itemize}
    \item  $\phi$ is a linear combination of finitely many step functions, and $\zeta_j$ be all the discontinuity points of $\phi$.
    \item $\gamma_i'$, $1\le i\le \bar N$ be all the different points $\gamma_i$ and $\zeta_j$, ordered from left to right.
    \item M be the maximal length of a \emph{primitive connection} $\gamma_j'= T^m(T\gamma_i')$, and $0$ if there is no such connection.
    \item $R$ be q minus the number of different primitive connections $\gamma_j= T^m(T\gamma_i')$.
    \item $\bar L$ be the language of the coding of $T$ by the points $\gamma_j'$, where the interval $[\gamma_j',\gamma_{j+1}')$ is coded by the symbol $A^{(i+1)}$, $1\le i\le \bar N-1$, $[0,\gamma_1')$ by $A^{(1)}$ and $[\gamma_{\bar N}',1)$ by $A^{(\bar N+1)}$.
    \item $\tilde\phi$ be the map associating to the symbol $A^{(i)}$ the value of the function $\phi$ on the interval coded by $A^{(i)}$.
\end{itemize}

We choose a vertex $w$ of the Rauzy graph $G_{M+1}$ of the language $\bar L$. Then it has exactly $R+1$ different return paths, whose labels are words $U_1,\ldots,U_{R+1}$. Let $\iota_{i,j}$ be the number of
occurrences of the symbol $A^{(i)}$ in the word $U_j$, $1\le i\le \bar N+1$, $1\le j\le R+1$. Let 
$$d_j=\sum_{i=1}^{\bar N+1}\iota_{i,j}\tilde\phi(A^{(i)}),\;\; 1\le j\le R+1.$$
Then the map $T_\phi$  is minimal if and only if $\gcd(q,d_1,...d_{R+1})=1$ (if $d_l = 0$, we consider it has a common factor with every integer).
\end{lemma}

\begin{proof}[Proof of Proposition \ref{prop:erg1}]
Recall that $h(x,k)=(x+2\varkappa,k+\phi(x))$. Without loss of generality, assume $\varkappa\in (0,\frac{1}{2})$, and $$\phi(x)=\begin{cases}
    0,&\text{if} \;x\in [0,1-2\varkappa);\\
    1,&\text{if}\; x\in [1-2\varkappa,1-\varkappa);\\
    -1,&\text{if}\; x\in [1-\varkappa,1).
\end{cases}$$The other case can be proven analogously, which we omit here.

We would like to apply Lemma \ref{lemma:criteria}. This is applicable just as those examples discussed in \cite{SP} (for example, Theorem 11 and Proposition 12 therein).

To start with,
we can think that the circle rotation $x\mapsto x+2\varkappa$ is an interval exchange transformation with two intervals. In particular, it has one discontinuity at $1-2\varkappa$. Now in order to apply the above criteria, we can check that $\gamma_1'=1-2\varkappa$, $\gamma_2'=1-\varkappa$, $\bar N=2$. 

We observe that, in this case there is no primitive connection, hence $M=0$, $R=\bar N-M=2$. This motivates us to study the Rauzy graph $G_1$. We will consider the case $\frac{1}{3}<\varkappa<\frac{1}{2}$, the other cases can be argued analogously. 

We code the circle $\mathbb S^1$ in the following way, $[0,1-2\varkappa)$ is coded by $A^{(1)}$, $[1-2\varkappa,1-\varkappa)$ by $A^{(2)}$, and $[1-\varkappa,1)$ by $A^{(3)}$.

The words of length 1 are $A^{(1)},A^{(2)},A^{(3)}$. The edge from $A^{(2)}$ to $A^{(1)}$ is labeled $A^{(2)}A^{(1)}$, from $A^{(1)}$ to $A^{(3)}$ is $A^{(1)}A^{(3)}$, from $A^{(3)}$ to $A^{(2)}$ is $A^{(3)}A^{(2)}$, from $A^{(2)}$ to $A^{(2)}$ is $A^{(2)}$ and from $A^{(3)}$ to $A^{(3)}$ is $A^{(3)}$. The three return paths of $A^{(2)}$ are $A^{(2)}$, $A^{(1)}A^{(3)}A^{(2)}$ and $A^{(1)}A^{(3)}A^{(3)}A^{(2)}$. Then it follows immediately that $d_1=1$, $d_2=0$ and $d_3=-1$. Therefore, for any $q\ge 1$, $\gcd(q+1,d_1,d_2,d_3)=1$. The minimality of $h$ follows.

Finally, since $h$ preserves the measure $\text{Leb}_{\mathbb S^1}\otimes \text{Count}$, by considering the ergodic decomposition and using the minimality, one can easily obtain the ergodicity. (Please note that, we do not claim unique ergodicity, which is much more difficult to prove.)
\end{proof}

\subsubsection{Recurrence}
Now we are ready to prove the main result Theorem \ref{thm:rec}, i.e. recurrence of the classical piecewise linear Fermi-Ulam models under resonance. Moreover, we provide an easy procedure to locate the escaping orbits (as well as bounded ones), though they are exceptionally rare on the infinite-volume phase cylinder.

\begin{thm}[Recurrence]\label{thm:rec}
    Under the classical resonance as in Definition \ref{def:resonance}, the escaping set of the piecewise linear Fermi-Ulam accelerator constitute a null set on the phase cylinder, and consequently the dynamics is recurrent in the sense that almost every orbit returns to its initial momentum level infinitely often.
\end{thm}

\begin{proof}
We recall from Remark \ref{rmk:onlyhighenergy} that the dynamics of the resonant piecewise linear Fermi-Ulam accelerator coincide with that of $P$ only for large energies due to the issue of re-collisions for small energies. However, the escaping set is exactly the transient part of a Fermi-Ulam accelerator, see \cite[Lemma 4.3]{Dol} and \cite[Corollary 6] {Zhou21}. Therefore it suffices to prove that the escaping set of $P$ is a null set on the $(\tau,I)$-phase cylinder, or equivalently, $P$ is recurrent.

Fix $q\in\mathbb{N}$. Suppose that $(A,B)$ are in $q$-resonance.
We decompose the $(\tau,I)$-phase cylinder into a collection of invariant circles $\{\mathcal{C}_D\}_{D\in[0,1)}$. It suffices to prove recurrence for irrational $D$'s, since the subset of all rational ones has zero measure on the phase-cylinder.

By Proposition \ref{prop:erg}, $F$ is ergodic with respect to the Lebesgue measure.
We recall from \eqref{skewprodtyp1}\eqref{skewprodtyp2} in Proposition \ref{prop:skewprod} that the momentum change $\eta$ only takes nonzero values $\pm1$ within the $q+1^{st}$ combined interval $\mathcal{C}_D^{m,0}\cup\mathcal{C}_D^{m,p+1}$, and it is straightforward to check that $\eta$ has zero average over this interval. 

Therefore the  dynamics of $P$ is a skew product with an ergodic base and a cocycle with zero average on the fiber, and hence by a result of \cite{atk76} such a skew product is recurrent. Therefore, the escaping set takes up at best a null set on an irrational circle $\mathcal{C}_D$. 
\end{proof}

We know from Theorem \ref{thm:rec} that escaping orbits are exceptionally rare. However, they do exist. We now describe an algorithm to locate an escaping orbit and also provide some explicit examples.

The collection of $\{\mathcal{C}_D\}_{D\in\mathbb{Q}}$ make up a null set in the $(\tau,I)$-phase cylinder. For $D$ rational, typical or atypical, it is easy to see that $F$ is eventually periodic. Now we describe the procedure to locate the escaping orbits (and also bounded ones) on the rational invariant circles $\mathcal{C}_D$, $D\in\mathbb{Q}$.

Fix $q\in\mathbb{N}$ and $D=r/s$ for $r,s$ relatively prime. Suppose that the parameters $(A,B)$ are in $q$-resonance. We denote $Q:=\lcm(s,q+1)$. For $(\tau_0,I_0)\in\mathcal{C}_D$ with $I_0\gg1$, there are at least two identical points $\tau_{n_1}=\tau_{n_2}$ in the first $Q$-iterates. And then the \emph{period momentum change} 
\begin{equation*}
    \Delta \eta (\tau_0) := \sum_{j=n_1}^{n_2 -1} \eta \, \circ F^j (\tau_0)
\end{equation*}
completely determines the behavior the orbit $\{(\tau_n,I_n)\}_n$: if $\Delta \eta (\tau_0)>0$, then $(\tau_0,I_0)$ produces an escaping orbit; if $\Delta \eta (\tau_0)=0$, then $(\tau_0,I_0)$ produces a bounded orbit.

\vspace{10pt}

We conclude the discussion of the classical piecewise linear Fermi-Ulam accelerators with the example of special resonance $q=1$, which in particular includes the very original Ulam's model $(A,B,T)=(1/\sqrt2,\sqrt2,1)$.

\begin{example}[Escaping and bounded orbits]\label{example:escbdd} The choice of parameters $(A,B,T)=(1/\sqrt2,\sqrt2,1)$ made by Ulam \cite{Ulam61} and Zharnitsky \cite{Zha98} is a special case of the special resonance $q=1$. In particular, the one-parameter family of linearly escaping orbits found by Zharnitsky \cite{Zha98} correspond to the 0$^{th}$ components $\{\mathcal{C}_{\frac{1}{2},n}^{0,0}\}_{n\in\mathbb{N}}$ (indexed by the floor number $n$) of the invariant circle $\mathcal{C}_{\frac{1}{2}}$. In fact, by Proposition \ref{prop:skewprod}, on $\mathcal{C}_{\frac{1}{2}}$, the base dynamics of the skew product is trivial, i.e. $F(\tau)=\tau$, and the period momentum change $\Delta\eta=\eta$ on the integer fibers takes value 1 on $\{\mathcal{C}_{\frac{1}{2},n}^{0,0}\}_{n\in\mathbb{N}}$ producing linearly escaping orbits, and value $0$ on $\{\mathcal{C}_{\frac{1}{2},n}^{0,1}\}_{n\in\mathbb{N}}$ and $\{\mathcal{C}_{\frac{1}{2},n}^{0,2+}\}_{n\in\mathbb{N}}$ producing bounded orbits.
\end{example}

\section{The Quantum Fermi-Ulam Models}\label{sec:QFU}

Hamiltonian systems in classical mechanics find their quantum counterparts through a standard process called ``quantization". Famous examples include, but not restricted to, the quantum kicked rotator \cite[Chapter 16]{Bour05}, \cite{ccif79,IzSe80}, the quantum billiards \cite{buzw05,hass10,haze04}, the quantum CAT map \cite{fndb03, keat91} and the arithmetic quantum unique ergodicity \cite{lind06,rusa94}. The readers may refer to \cite{Folland89,Hom07,stein93} and the references therein for detailed exposition on the subject.

In this section, we consider the quantized piecewise linear Fermi-Ulam accelerators via the standard Kohn-Nirenberg quantization process. Just as Hamiltonian ODE is to a classical system, the evolution of the quantum Fermi-Ulam accelerator is governed by the following Schr\"odinger equation 
\begin{equation}\label{eq:schrod}
    i\hbar \dfrac{d\varphi}{dt} = \mathcal{H}_{\hbar} \varphi 
\end{equation}
and we choose the Dirichlet boundary condition 
\begin{equation}\label{eq:dirichletbd}
    \varphi(0,t) = \varphi(l(t),t) = 0
\end{equation}
where $x$ indicates the horizontal position of the particle with the fixed wall positioned at $x=0$, $\hbar$ is the reduced Plank constant and $\mathcal{H}_{\hbar}=-\dfrac{\hbar^2}{2}\Delta$ is the quantum Hamiltonian operator acting on $L^2\big((0,l(t))\big)$.

Throughout this paper, we take, for simplicity of notation and computation, $\hbar=1$.\\

The goal of this section is to show in Theorem \ref{thm:quadengrow} that the quantum resonant accelerators enjoy quadratic energy growth in general, and in Theorem \ref{thm:quasienspec} that the quasi-energy spectra are absolutely continuous with finitely many continuity components. Furthermore, we reveal the direct connection between the energy growth and the shape of the quasi-energy spectra with explicit formulas. Our results cover and greatly generalize that of \cite{Seba90} as a special resonance case (c.f. Example \ref{example:11resonance}).

We begin in \S \ref{ssec:flo} with the derivation of the quantum Floquet operators (c.f. \eqref{eq:1floquet}\eqref{eq:2floquet}), which describe the wave propagation in one period. In particular in \S\ref{sssec:rfm} we show that under general resonance the information of the Floquet operators can be encoded into Floquet matrices of finite dimensions (c.f. Proposition \ref{prop:floqmatrix}). Then in \S\ref{ssec:egr} we show in Theorem \ref{thm:quadengrow} that the quantum resonance leads to quadratic energy growth. Finally in \S\ref{ssec:qes} we prove in Theorem \ref{thm:quasienspec} that absolute continuity of the quasi-energy spectra under resonance and also reveal through explicit formulas the direction connection between the energy growth and the shape of the quasi-energy spectra.

\subsection{The Floquet Operators}\label{ssec:flo}

In this subsection we derive the quantum Floquet operators (c.f. \eqref{eq:1floquet}\eqref{eq:2floquet}), describing the wave propagation of \eqref{eq:schrod} in one period. More precisely, first in \S\ref{sssec:swt} we transfer the time dependence in boundary conditions \eqref{eq:dirichletbd} to time-dependent potentials in the Schr\"odinger equations, then in \S\ref{sssec:fo} we derive the Floquet wave propagators, and lastly in \S\ref{sssec:rfm} we reduce the quantum Floquet operators to Floquet matrices of finite dimensions (c.f. Proposition \ref{prop:floqmatrix}).

\subsubsection{Stopping the Wall Transformation}\label{sssec:swt}

The time dependence in the boundary condition \eqref{eq:dirichletbd} triggers annoyance for the study of the equation and the notion $L^2((0,l(t)))$ itself is also technically vague. To handle such nuisance, we perform change of coordinates to ``stop the moving wall", at the price of introducing artificial background potential to the system. This technique was used to handle the moving boundary in heat equation and is closely related to the Liouville transform \cite{Bell53}, and has played a powerful role in the study of the classical Fermi-Ulam accelerators \cite{Zha98} as well the quantum ones \cite{Seba90} from whom we inherit.\\

Let $U(t)$ be the wave propagator of \eqref{eq:schrod}. 

$\tilde{U}=W^{-1}U$ where $W(t):L^2\big((0,1)\big)\to L^2\big((0,l(t))\big)$ via 
\begin{equation}\label{eq:stopwall}
    (W(t)^{-1}\phi)(x) := \sqrt{l(t)} \exp\left( -\dfrac{1}{4}l(t)\dot{l}(t)x^2 \right) f(l(t)x), \; \forall f\in L^2\big(0,l(t)\big).
\end{equation}

The quantum Fermi-Ulam accelerator \eqref{eq:schrod} now turns into the following Schr\"odinger equation with an artificial time-dependent background potential and time-independent boundary condition 
\begin{equation}\label{eq:schrodnew}
    \begin{cases}
        & i\dfrac{d\tilde{\varphi}}{dt} = \dfrac{1}{l(t)^2} \left( \mathcal{H}_0 + \dfrac{x^2}{4}l(t)^3 \ddot{l}(t) \right) \tilde{\varphi} \\
        & \tilde{\varphi}(0,t) = \tilde{\varphi}(1,t) = 0
    \end{cases}
\end{equation}
where $\mathcal{H}_0=-\Delta$ is the Laplacian acting on $L^2\big((0,1)\big)$. 

We introduce ``new time" $\zeta = g(t) :=\displaystyle\int_0^t\dfrac{ds}{l(s)^2}$ to get rid of the term $\dfrac{1}{l(t)^2}$ in front of $\mathcal{H}_0$ in \eqref{eq:schrodnew}, so the final equation in new time derivative presents 
\begin{equation}\label{eq:quantumulam}
    \begin{split}
        & i\dfrac{d\tilde{\varphi}}{d\zeta} = \left( \mathcal{H}_0 + \dfrac{x^2}{4}l(g^{-1}(\zeta))^3 \ddot{l}(\zeta) \right) \tilde{\varphi} \\
        & \tilde{\varphi}(0,t) = \tilde{\varphi}(1,t) = 0
    \end{split}
\end{equation}

Since $l$ is piecewise linear, $\ddot{l}=0$ except at two turning moments $t=0,T$. Therefore \eqref{eq:quantumulam} becomes 
\begin{equation}\label{eq:quantumlinulam}
    \begin{split}
        & i\dfrac{d\tilde{\varphi}}{d\zeta} = \bigg( \mathcal{H}_0 + x^2 \big(J_1 \delta_{2\mathcal{T}}(\zeta) - J_2 \delta_{2\mathcal{T}}(\zeta - \mathcal{T})\big) \bigg) \tilde{\varphi} \\
        & \tilde{\varphi}(0,t) = \tilde{\varphi}(1,t) = 0
    \end{split}
\end{equation}
where $J_1=\dfrac{B(A-B)}{2T}, J_2=\dfrac{A(A-B)}{2T}$ are two constants, and the subscript $2\mathcal{T}$ in the Dirac-Delta indicates that the jumps are $2\mathcal{T}$-periodic in time, i.e. $\delta_{2\mathcal{T}}(\zeta + 2\mathcal{T})=\delta_{2\mathcal{T}}(\zeta)$.

\subsubsection{The Floquet Operators}\label{sssec:fo}
Now we derive the Floquet operators which describe the evolution of the quantum Fermi-Ulam accelerator after one period $\Delta t=2$. We drop the tilde over the wave and use $\varphi$ for solution of \eqref{eq:quantumlinulam} for ease of notation.

First we observe that $\bigg\{\big((n\pi)^2,\sin(n\pi x)\big)\bigg\}_{n\in\mathbb{N}}$ is an orthogonal family of eigenpairs of $\mathcal{H}_0$ in $L^2\big((0,1)\big)$. We consider the formal series  
\begin{equation}\label{eq:ansatzinitialwave}
    \varphi(x,\zeta) = \sum_{n=1}^{\infty} A_n(\zeta) \sin(n\pi x).
\end{equation}

At the first jump moment $\zeta=0$, 
\begin{equation}\label{eq:jump0}
    i\dfrac{d\varphi}{d\zeta} = J_1 x^2 \delta(\zeta) \varphi \; \implies \; \varphi(x,0+) = \varphi(x,0-) \exp{(-i J_1 x^2)}.
\end{equation}

During the free evolution $\zeta\in(0,\mathcal{T})$, 
\begin{equation}\label{eq:freeevol}
    i\dfrac{d\varphi}{d\zeta} = -\Delta \varphi \; \implies \varphi(x,\mathcal{T}-) = \sum_{n=1}^{\infty} \exp{(-i(n\pi)^2\mathcal{T})} A_n(0+) \sin(n\pi x)
\end{equation}

Combining \eqref{eq:jump0} and \eqref{eq:freeevol}, we obtain the first-half Floquet operator $\mathcal{F}_1$ sending an initial wave $\varphi_0(x)=\varphi(x)$ at time $\zeta=0-$ to $\varphi_{\mathcal{T}}(x)$ at time $\zeta=\mathcal{T}-$ via 
\begin{equation}\label{eq:1floquet}
    \varphi_{\mathcal{T}}(x) = \exp{(-i J_1 x^2)} \sum_{n=1}^{\infty} \exp{(-i(n\pi)^2\mathcal{T})} A_n \sin(n\pi x)
\end{equation}
where $A_n = \displaystyle\int_0^1 \varphi(x) \sin(n\pi x) dx$.

Similarly, the second-half Floquet operator $\mathcal{F}_2$ sending an initial wave $\varphi_{\mathcal{T}}(x) \in L^2\big((0,1)\big)$ at time $\zeta=\mathcal{T}-$ to the wave $\varphi_{2\mathcal{T}}(x)$ at time $\zeta=2\mathcal{T}-$ via 
\begin{equation}\label{eq:2floquet}
    \varphi_{2\mathcal{T}}(x) = \exp{(i J_2 x^2)} \sum_{n=1}^{\infty} \exp{(-i(n\pi)^2\mathcal{T})} B_n \sin(n\pi x)
\end{equation}
where $B_n = \displaystyle\int_0^1 \varphi_{\mathcal{T}}(x) \sin(n\pi x) dx$.

\subsubsection{The Resonant Floquet Matrices}\label{sssec:rfm}
Here we show that the Floquet operators from the previous section can be represented by finite-dimensional Floquet matrices (c.f. Proposition \ref{prop:floqmatrix}), provided that the parameters $A,B,T$ are in resonance.

\begin{definition}[Quantum resonance]\label{def:quantumresonanace}
    For $p,q\in\mathbb{N}$ relatively prime, we say the parameters $(A,B,T)$ are in \emph{(quantum) $(p,q)$-resonance} if $\dfrac{\pi}{2}\dfrac{T}{AB}=\dfrac{p}{q}$.
\end{definition}

We reveal the hidden periodicity due to resonance in the expressions of \eqref{eq:1floquet}\eqref{eq:2floquet} and represent the Floquet operators $\mathcal{F}_1,\mathcal{F}_2$ with finite-dimensional matrices.

Fix $p,q\in\mathbb{N}$ relatively prime. Suppose that $(A,B,T)$ are in $(p,q)$-resonance. Then \eqref{eq:1floquet} becomes 
\begin{align*}
    \varphi_{\mathcal{T}}(x) &= \exp{(-i J_1 x^2)} \sum_{n=1}^{\infty} \exp{(-2\pi i n^2 p/q)} A_n \sin(n\pi x) \\
    &= \exp{(-i J_1 x^2)} \sum_{m=1}^{q-1} \exp{(-2\pi i m^2 p/q)} \underbrace{\sum_{l=0}^{\infty} A_{lq+m} \sin \big((lq+m)\pi x \big)}_{= :D_m}
\end{align*}

We will show that $D_m$ ($m=1,\cdots,q-1$) can be represented by a finite sum. 

Inspired by the technique in handling of quantum kicked rotator \cite{IzSe80}, we consider a basis of initial waves $\{(\varphi(x+2n/q))\}_{n=0}^{q-1}$ rather than a single wave $\varphi$. In the case of kicked rotator, the operation is well-defined as $x\in\mathbb{S}^1$; however, here in the case of Fermi-Ulam accelerators the particle is physically confined between two walls and hence $x\in(0,1)$ after stopping the wall transformation, and hence $x+2n/q$ may exceed the prescribed domain. We resolve the issue by first extending the domain of $\varphi$ to $(-1,1)$ by symmetry and then to the entire $\mathbb{R}$ by periodicity: 
\begin{equation*}
    \varphi(x) = -\varphi(-x), \; x\in(-1,0); \; \varphi(x+2)=\varphi(x), x\in\mathbb{R}.
\end{equation*}
It is easy to verify that such extension is compatible with the evolution of \eqref{eq:quantumlinulam}, as it respects the parity of $\sin$ and it commutes with the time evolution of \eqref{eq:quantumlinulam}.

Now we are prepared to simplify $D_m$.

We observe that 
\begin{align*}
    \cos\left(\frac{2mn\pi}{q}\right)D_m & = \sum_{l=0}^{\infty} A_{m+ql}\sin\left((m+ql)\pi x\right)\cos\left(\frac{2mn\pi}{q}\right) \\
        & = \sum_{l=0}^{\infty} A_{m+ql}\sin\left((m+ql)\pi x\right)\cos\left(\frac{2(m+ql)n\pi}{q}\right) \\
        & = \frac{1}{2}\sum_{l=0}^{\infty} A_{m+ql}\left[\sin\left((m+ql)\pi\left( x+\frac{2n}{q}\right)\right) \right.\\
        & \quad \left. - \sin\left((m+ql)\pi\left( x-\frac{2n}{q}\right)\right)\right].
\end{align*}

Summing over $m$ we then obtain
\begin{equation}\label{eq:dmeq1}
    \sum_{m=0}^{q-1}\cos\left(\frac{2mn\pi}{q}\right)D_m = \frac{1}{2}\left(\varphi\left(x+\frac{2n}{q}\right)+\varphi\left(x-\frac{2n}{q}\right)\right).
\end{equation}

We reversely solve for $D_m$ out of \eqref{eq:dmeq1} by multiplying both sides of \eqref{eq:dmeq1} with $\cos(2kn\pi/q)$ and summing over $n$, and then we obtain 
\begin{equation}\label{eq:dmrelation}
    D_{m}+D_{q-m}  = \frac{1}{q} \sum_{n=0}^{q-1} \left(\varphi\left(x+\frac{2n}{q}\right) + \varphi\left(x-\frac{2n}{q}\right)\right)\cos\left(\frac{2\pi n m}{q}\right)
\end{equation}
where $0\le m\le q$ and $B_q=B_0$, and we have used the fact that 
\begin{align*}
    \sum_{k=0}^{q-1}\cos\left(\frac{2\pi mk}{q}\right)\cos\left(\frac{2\pi nk}{q}\right) = \left\{
    \begin{array}{cc}
        q & m=n=0 \\
        q/2 & m=n\neq 0 \, \hbox{ or }\, m+n=q\\
        0 & \hbox{otherwise}
    \end{array}
    \right. .
\end{align*}

We update the expression of the Floquet operator $\mathcal{F}_1$ with simplified $D_m$ relation \eqref{eq:dmrelation}. We denote $\tilde{x}:=x \pmod{2}$ taking value in $(-1,1)$. We compute 
\begin{align*}
     \varphi_{\mathcal{T}}(x) 
      & =\exp{(-i J_1 \tilde{x}^2)} \sum_{m=0}^{q-1} e^{-2\pi i m^2 p/q} D_m \\
      & = \frac{1}{2} \exp{(-i J_1 \tilde{x}^2)} \left(\sum_{m=0}^{q-1} e^{-2\pi i m^2 p/q} D_m + \sum_{m=0}^{q-1} e^{-2\pi i(q-m)^2 p/q} D_{q-m}\right) \\
      & = \frac{1}{2} \exp{(-i J_1 \tilde{x}^2)} \left(\sum_{m=0}^{q-1} e^{-2\pi i m^2 p/q} D_m + \sum_{m=0}^{q-1} e^{-2\pi i m^2 p/q} D_{q-m}\right) \\
      & = \frac{1}{2q} \exp{(-i J_1 \tilde{x}^2)} \sum_{m=0}^{q-1}\sum_{n=0}^{q-1} e^{-2\pi i m^2 p/q}  \\&\quad\quad\quad  \times \left(\varphi\left(x+\frac{2n}{q}\right) + \varphi\left(x-\frac{2n}{q}\right)\right)\cos\left(\frac{2\pi n m}{q}\right) \\
      &= \frac{1}{2} \exp{(-i J_1 \tilde{x}^2)} \sum_{n=0}^{q-1} \left( \varphi\left(x+\frac{2n}{q}\right) + \varphi\left(x-\frac{2n}{q}\right) \right) \\& \quad\quad\quad \times \underbrace{\frac{1}{q} \sum_{m=0}^{q-1}  e^{-2\pi i m^2 p/q} \cos\left(\frac{2\pi n m}{q}\right)}_{=:\gamma_n} .
\end{align*}

We observe that $\gamma_n$ is cyclic in $q$ ($\gamma_{n+q}=\gamma_n$) and symmetric ($\gamma_{q-n}=\gamma_{-n}=\gamma_n$). We also recall that $\varphi(x+2)=\varphi(x)$. We proceed 
\begin{align*}
    \varphi_{\mathcal{T}}(x) 
    &=  \frac{1}{2} \exp{(-i J_1 \tilde{x}^2)} \sum_{n=0}^{q-1} \gamma_n \left( \varphi\left(x+\frac{2n}{q}\right) + \varphi\left(x-\frac{2n}{q}\right) \right) \\
    &= \frac{1}{2} \exp{(-i J_1 \tilde{x}^2)} \left( \sum_{n=0}^{q-1} \gamma_n \varphi\left(x+\frac{2n}{q}\right) + \sum_{n=0}^{q-1} \gamma_{n} \varphi\left(x+\frac{2(q-n)}{q}\right) \right) \\
    &= \frac{1}{2} \exp{(-i J_1 \tilde{x}^2)} \left( \sum_{n=0}^{q-1} \gamma_n \varphi\left(x+\frac{2n}{q}\right) + \sum_{n=0}^{q-1} \gamma_{q-n} \varphi\left(x+\frac{2n}{q}\right) \right)\\
    &= \frac{1}{2} \exp{(-i J_1 \tilde{x}^2)} \sum_{n=0}^{q-1} \varphi\left(x+\frac{2n}{q}\right) (\gamma_n + \gamma_{q-n})\\
    &= \exp{(-i J_1 \tilde{x}^2)} \sum_{n=0}^{q-1} \gamma_n \varphi\left(x+\frac{2n}{q}\right) .
\end{align*}

Therefore we obtain a simplified expression of the first-half Floquet operator $\mathcal{F}_1$ 
\begin{equation}\label{eq:simpfloq1}
    \varphi_{\mathcal{T}}(x) = \exp{(-i J_1 \tilde{x}^2)} \sum_{n=0}^{q-1} \gamma_n \varphi\left(x+\frac{2n}{q}\right).
\end{equation}

Similarly, we have for the second-half Floquet operator $\mathcal{F}_2$ 
\begin{equation}\label{eq:simpfloq2}
    \varphi_{2\mathcal{T}}(x) = \exp{(i J_2 \tilde{x}^2)} \sum_{n=0}^{q-1} \gamma_n \varphi_{\mathcal{T}}\left(x+\frac{2n}{q}\right).
\end{equation}

We summarize the results of all the computation hereinbefore into the following proposition.

\begin{prop}[Resonant Floquet matrices]\label{prop:floqmatrix}
Suppose $p,q\in\mathbb{N}$ are relatively prime. Assume that the parameters $(A,B,T)$ are in $(p,q)$-resonance as in Definition \ref{def:quantumresonanace}. Then the evolution of \eqref{eq:quantumlinulam} is governed by two $q$-dimensional Floquet matrices $S(x)=(S_{mn}(x))=(\alpha_m(x)\gamma_{n-m}))$ and $R(x)=(R_{mn}(x))=(\beta_m(x)\gamma_{n-m}))$, where $\alpha_m(x)=\exp{(-i J_1 (\widetilde{x+2m/q})^2)}$ and $\beta_m(x)=\exp{(i J_2 (\widetilde{x+2m/q})^2)}$. 

More precisely, for any initial wave vector $\Phi(x)=\big( \varphi(x),\varphi(x+2/q),\cdots,(x+2(q-1)/q) \big)$, the first-half period evolution of \eqref{eq:quantumlinulam} is described by 
\begin{equation*}
    \Phi_{\mathcal{T}}(x)= S(x) \Phi(x),
\end{equation*}
where $\Phi_{\mathcal{T}}(x)=\big( \varphi_{\mathcal{T}}(x),\varphi_{\mathcal{T}}(x+2/q),\cdots,\varphi_{\mathcal{T}}(x+2(q-1)/q) \big)$ and 
\begin{equation*}
S(x) = 
\begin{pmatrix}
        \alpha_0 (x) & & &\\
        &\alpha_1(x) & &\\
        & &\cdots & \\
        & & &\alpha_{q-1}(x)
    \end{pmatrix}
    \begin{pmatrix}
        \gamma_0 &\gamma_1 &\cdots &\gamma_{q-1}\\
        \gamma_{-1} &\gamma_0 &\cdots &\gamma_{q-2}\\
        & &\cdots & \\
        \gamma_{-(q-1)} &\gamma_{-(q-2)} &\cdots &\gamma_0
    \end{pmatrix} \; ;
\end{equation*}
and the second-half period evolution is described by 
\begin{equation*}
    \Phi_{2\mathcal{T}}(x)= R(x) \Phi_{\mathcal{T}}(x)
\end{equation*}
where $\Phi_{2\mathcal{T}}(x)=\big( \varphi_{2\mathcal{T}}(x),\varphi_{2\mathcal{T}}(x+2/q),\cdots,\varphi_{2\mathcal{T}}(x+2(q-1)/q) \big)$, and 
\begin{equation*}
R(x) = 
\begin{pmatrix}
        \beta_0 (x) & & &\\
        &\beta_1(x) & &\\
        & &\cdots & \\
        & & &\beta_{q-1}(x)
    \end{pmatrix}
    \begin{pmatrix}
        \gamma_0 &\gamma_1 &\cdots &\gamma_{q-1}\\
        \gamma_{-1} &\gamma_0 &\cdots &\gamma_{q-2}\\
        & &\cdots & \\
        \gamma_{-(q-1)} &\gamma_{-(q-2)} &\cdots &\gamma_0
    \end{pmatrix} \; .
\end{equation*}
\end{prop}

\subsection{Energy Growth under Resonance}\label{ssec:egr}
In this subsection we show that the resonant piecewise linear quantum Fermi-Ulam accelerators enjoy quadratic energy growth for most, if not all, initial waves. This phenomenon of quantum acceleration is in sharp contrast to its classical origin where recurrence  is typical. This result covers and greatly generalizes the special resonant case considered in \cite{Seba90}. This is, to the authors' best knowledge, the second example (with mathematical proof) next to the kicked rotator \cite{ccif79,IzSe80} that a quantum system differs substantially from its classical counterpart, and the exceptionally rare event of acceleration due to resonance in the classical setting gets much amplified in the quantized system rather than being suppressed as predicted by Anderson localization \cite{fgp82}.\\ 

We study the \emph{energy} of the quantum Fermi-Ulam accelerator \eqref{eq:quantumlinulam} at $N$ periods:
\begin{equation}\label{eq:defenergy}
    \begin{split}
        E(N) &= -\dfrac{1}{2q} \la\Phi_{2N\mathcal{T}}, \Delta \Phi_{2N\mathcal{T}}\ra\\
        & = -\dfrac{1}{2q} \sum_{m=0}^{q-1} \int_0^1 \overline{\varphi}_{2N\mathcal{T}} (x+2m/q) \Delta \varphi_{2N\mathcal{T}} (x+2m/q) \, dx.
    \end{split}
\end{equation}

\begin{thm}[Quadratic energy growth]\label{thm:quadengrow}
Suppose $p,q\in\mathbb{N}$ are relatively prime. Assume that the parameters $(A,B,T)$ are in $(p,q)$-resonance as in Definition \ref{def:quantumresonanace}. Then for any $N\in\mathbb{N}$, the energy of the quantum Fermi-Ulam accelerator \eqref{eq:quantumlinulam} at $N^{th}$ period is given by 
\begin{equation}
    E(N) = \mathfrak{a} N^2 + \mathfrak{b} N + \mathfrak{c}
\end{equation}
where $\mathfrak a,\mathfrak b,\mathfrak c$ are constants, and in particular, $\mathfrak{a} \ge 0$.
\end{thm}

\begin{proof}
By Proposition \ref{prop:floqmatrix}, the evolution of \eqref{eq:quantumlinulam} after $N$ complete periods is given by 
\begin{equation*}
    \Phi_{2N\mathcal{T}}(x) = (RS)^N \Phi(x).
\end{equation*}

We observe that $R,S$ are two unitary matrices (which unfortunately do not commute), and hence the eigenvalues of the matrix $RS$ are in the form 
\begin{equation}\label{eq:eigenvalue}
    \lambda_j (x) = \exp{(i \xi_j(x))}, \; j=0,\cdots,q-1
\end{equation}
and $RS=Q\Lambda Q^{-1}$, where $\Lambda(x)=\diag\big(\lambda_j(x)\big)$ and $Q(x)=(Q_{mn}(x))$ is also unitary. Therefore 
\begin{equation*}
    \Phi_{2N\mathcal{T}}(x) = Q \Lambda^N Q^{-1} \Phi(x) = Q \diag\big(e^{i N \xi_j(x)}\big) Q^{-1} \Phi(x).
\end{equation*}
Therefore the energy after $N$ periods is
\begin{align*}
E(N) 
&= -\dfrac{1}{2q} \la Q \Lambda^N Q^{-1} \Phi, \Delta(Q \Lambda^N Q^{-1} \Phi)\ra \\
&= \dfrac{N^2}{2q} \la Q \Lambda^N Q^{-1} \Phi, Q \diag\big( (\xi')^2 e^{iN\xi} \big) Q^{-1} \Phi\ra \\ 
&\quad -\dfrac{iN}{2q} \la Q \Lambda^N Q^{-1} \Phi, Q \diag\big( \xi'' e^{iN\xi} \big) Q^{-1} \Phi\ra \\
&\quad - \dfrac{iN}{q}\la Q \Lambda^N Q^{-1} \Phi, Q' \diag\big(\xi'e^{iN\xi}\big) Q^{-1} + Q \diag\big(\xi'e^{iN\xi}\big) (Q^{-1} \Phi)'\ra\\
&\quad - \dfrac{1}{2q} \la Q \Lambda^N Q^{-1} \Phi, Q'' \Lambda^N Q^{-1} \Phi + Q \Lambda^N (Q^{-1} \Phi)''\ra \\
&\quad -\dfrac{1}{q}\la Q \Lambda^N Q^{-1} \Phi, Q' \Lambda^N (Q^{-1} \Phi)'\ra \\
&= \dfrac{N^2}{2q} \la Q^{-1} \Phi, \diag\big( (\xi')^2 \big) Q^{-1} \Phi\ra -\dfrac{iN}{2q} \la Q^{-1} \Phi, \diag\big( \xi'' \big) Q^{-1} \Phi\ra \\ 
&\quad - \dfrac{iN}{q}\la Q \Lambda^N Q^{-1} \Phi, Q' \diag\big(\xi'e^{iN\xi}\big) Q^{-1}\ra - \dfrac{iN}{q}\la Q^{-1} \Phi , Q \diag\big(\xi'\big) (Q^{-1} \Phi)'\ra\\
&\quad - \dfrac{1}{2q} \la Q \Lambda^N Q^{-1} \Phi, Q'' \Lambda^N Q^{-1} \Phi\ra \underbrace{- \dfrac{1}{2q} \la Q^{-1} \Phi, Q^{-1} \Phi''\ra }_{E(0)} \\
&\quad - \dfrac{1}{2q} \la Q^{-1} \Phi, (Q^{-1})'' \Phi + 2(Q^{-1})' \Phi'\ra -\dfrac{1}{q}\la Q \Lambda^N Q^{-1} \Phi, Q' \Lambda^N (Q^{-1} \Phi)'\ra.
\end{align*}

Therefore the coefficient $a$ of the quadratic term is 
\begin{equation}\label{eq:quadcoeff}
    \mathfrak{a} = \dfrac{1}{2q}\la Q^{-1} \Phi, \diag\big( (\xi')^2 \big) Q^{-1} \Phi\ra
\end{equation}
which is a non-negative definite quadratic form, and hence $a\ge0$. 

The coefficients for the linear and the constant terms are respectively 
\begin{equation}\label{eq:lincoeff}
    \begin{aligned}
        \mathfrak{b} &= -\dfrac{i}{2q} \la Q^{-1} \Phi, \diag\big( \xi'' \big) Q^{-1} \Phi\ra - \dfrac{i}{q}\la Q \Lambda^N Q^{-1} \Phi, Q' \diag\big(\xi'e^{i\xi}\big) Q^{-1}\ra\\
    &\quad - \dfrac{iN}{q}\la Q^{-1} \Phi , Q \diag\big(\xi'\big) (Q^{-1} \Phi)'\ra ,
    \end{aligned}
\end{equation}
and 
\begin{equation}\label{eq:constcoeff}
    \begin{aligned}
       \mathfrak{c} &= E(0) - \dfrac{1}{2q} \la Q \Lambda^N Q^{-1} \Phi, Q'' \Lambda^N Q^{-1} \Phi\ra -\dfrac{1}{q}\la Q \Lambda^N Q^{-1} \Phi, Q' \Lambda^N (Q^{-1} \Phi)'\ra \\
    &\quad - \dfrac{1}{2q} \la Q^{-1} \Phi, (Q^{-1})'' \Phi + 2(Q^{-1})' \Phi'\ra . 
    \end{aligned}
\end{equation}
\end{proof}

\subsection{The Quasi-Energy Spectrum}\label{ssec:qes}
In this subsection we reveal the direct link between the energy growth and the shape of the quasi-energy spectrum and we provide explicit formulas for such link.\\

Firstly, we recall the concept of quasi-energy spectrum in quantum Floquet theory. 

A time-periodic quantum Hamiltonian system possesses an orthogonal set $\{\psi_j\}_j$ of solutions in the form of 
\begin{equation*}
    \psi_j(t) = \phi_j(t)e^{-it\rho_j}
\end{equation*}
where $\phi$ is time-periodic and $\rho$ is a real number. 

We observe that only a constant pops out in the front after a period evolution for such solutions: 
\begin{equation}\label{eq:quasienergydef}
    \psi_j(t+T) = \phi_j(t+T) e^{-i(t+T)\rho_j} =  e^{-iT\rho_j} \phi_j(t) e^{-it\rho_j} = e^{-iT\rho_j} \psi_j(t).
\end{equation}

In quantum Floquet theory, $\rho$ is called the \emph{quasi-energy} of the quantum system, and the collection of all quasi-energies is called the \emph{quasi-energy spectrum}.\\ 

In general it is very difficult to compute the quasi-energy of a given quantum system, and estimates of the shape or the size of quasi-energy spectrum often demand at least moderate regularity of the system \cite{bgmr21,grya2000,liuyuan10}, which is not fulfilled in our case due to the $\delta$-jumps in time. However, we are not in total desperation, as the jumps are only present for very specific moments. This phenomenon of finite ``kicks'' also arises in the (in)famous example of kicked rotators \cite{ccif79,IzSe80}, and inspired by the techniques therein from the physics community \cite{IzSe80}, we are able to compute the quasi-energy and the the quasi-energy spectrum of the piecewise linear Fermi-Ulam accelerators under resonance.

\begin{thm}[Quasi-energy spectrum]\label{thm:quasienspec}
    Suppose $p,q\in\mathbb{N}$ are relatively prime. Assume that the parameters $(A,B,T)$ are in $(p,q)$-resonance as in Definition \ref{def:quantumresonanace}. The quasi-energy $\rho$ of the quantum piecewise linear Fermi-Ulam accelerator \eqref{eq:quantumlinulam} is given by 
    \begin{equation}
        \rho_j(x_0) = -\dfrac{1}{2\mathcal{T}} \xi_j(x_0), \; x_0\in(0,1), \; j=0,\dots,q-1
    \end{equation}
    where $\xi_j$'s are the arguments of the eigenvalues $\lambda_j$'s of the resonant Floquet matrix $RS$ as in \eqref{eq:eigenvalue}.
\end{thm}

Theorem \ref{thm:quasienspec} shows the quasi-energy spectrum of a $(p,q)$-resonant quantum piecewise linear Fermi-Ulam accelerator has at most $q$ components $\big\{\rho_j(x) \,| \, x\in(0,1)\big\}_{j=0}^{q-1}$. These components are absolutely continuous unless $\xi_j(x)$'s are \emph{degenerate}, i.e. $\xi_j(x) \equiv \xi_j$ for some constant $\xi_j$, and in this case the corresponding spectrum component $\big\{\rho_j(x) \,| \, x\in(0,1)\big\}$ degenerates to a point.

Meanwhile, we recall from the proof of Theorem \ref{thm:quadengrow} that the quadratic and linear coefficients $\mathfrak{a},\mathfrak{b}$ responsible for the energy growth are directly proportional to the derivatives of $\xi$ (c.f.\eqref{eq:quadcoeff} \eqref{eq:lincoeff}). Therefore degenerate $\xi_j(x)$'s eliminate the quadratic and linear energy growth for certain initial waves, i.e. those consisting of the eigenfunctions of the degenerate eigenvalues. If the system is \emph{totally degenerate}, i.e. $\xi_j(x) \equiv \xi_j$ for all $j$, then the energy remains bounded and the quasi-energy spectrum degenerates to pure point. 

\begin{rmk}
    We note that such degeneracy (partial or total) is exceptionally rare, if not completely impossible, as they form a subspace  of co-dimension at least one in all possible choices of resonant parameters. 
\end{rmk}

We provide some examples in which we are able to obtain explicit formulas of the energy growth and the quasi-energy spectra.

\begin{example}[$(1,1)$-resonance]\label{example:11resonance}
    For $p=q=1$, the computation reduces to one dimension and there is no need to resort to Floquet matrices. In fact, the Floquet operators \eqref{eq:1floquet} \eqref{eq:2floquet} in $1:1$ resonance simplify into 
    \[
    \varphi_{\mathcal{T}} (x) = e^{-i J_1 x^2} \varphi(x), \; \varphi_{2\mathcal{T}}(x) = e^{i J_2 x^2} \varphi_{\mathcal{T}}(x) \; .
    \]
    Therefore the leading quadratic coefficient for the energy growth for an initial wave $\varphi$ is  
    \begin{equation*}
        a(1,1) = 2(J_2 - J_1)^2 \int_0^1 x^2 \bar{\varphi} \varphi \; dx >0.
    \end{equation*}
    And the quasi-energy spectrum has one continuous component 
    \begin{equation*}
        \spec(1,1) = \left\{ \dfrac{(J_1 - J_2)\tilde{x}^2}{2\mathcal{T}},\; x\in(0,1) \right\} = \left(0, \dfrac{J_1 - J_2}{2\mathcal{T}} \right) \; .
    \end{equation*}
    
    This is exactly the case studied in \cite{Seba90}; our computation agrees with those therein (with minor typos in \cite{Seba90} fixed).
\end{example}

\begin{example}[$(1,2)$-resonance]
    For $p=1,q=2$, the resonant Floquet matrices are respectively 
    \begin{equation*}
        S = 
        \begin{pmatrix}
            e^{-i J_1 \tilde{x}^2} & 0\\
            0 & e^{-i J_1 \widetilde{(x+1)}^2}
        \end{pmatrix}
        \begin{pmatrix}
            0 & 1\\
            1 & 0
        \end{pmatrix}
        = 
        \begin{pmatrix}
            0 & e^{-i J_1 \tilde{x}^2}\\
            e^{-i J_1 \widetilde{(x+1)}^2} & 0
        \end{pmatrix}
    \end{equation*}
    
    and 
    
    \begin{equation*}
        R = 
        \begin{pmatrix}
            e^{i J_2 \tilde{x}^2} & 0\\
            0 & e^{i J_2 \widetilde{(x+1)}^2}
        \end{pmatrix}
        \begin{pmatrix}
            0 & 1\\
            1 & 0
        \end{pmatrix}
        = 
        \begin{pmatrix}
            0 & e^{i J_2 \tilde{x}^2}\\
            e^{i J_2 \widetilde{(x+1)}^2} & 0
        \end{pmatrix},
    \end{equation*}
    
    and consequently 
    
    \begin{equation*}
        RS = 
        \begin{pmatrix}
            e^{i J_2 \tilde{x}^2 - i J_1 \widetilde{(x+1)}^2} & 0\\
            0 & e^{i J_2 \widetilde{(x+1)}^2 - i J_1 \tilde{x}^2}
        \end{pmatrix}.
    \end{equation*}
    
    The quasi-energy spectrum has one continuous component (in fact two overlapping components) 
    \begin{align*}
        \spec(1,2) &= \left\{ \dfrac{J_1 \widetilde{(x+1)}^2 - J_2\tilde{x}^2}{2\mathcal{T}},\; x\in(0,1) \right\} \cup \left\{ \dfrac{J_1 \tilde{x}^2 - J_2\widetilde{(x+1)}^2}{2\mathcal{T}},\; x\in(0,1) \right\}\\
        &= \left(-\dfrac{J_2}{2\mathcal{T}}, \dfrac{J_1}{2\mathcal{T}} \right) \cup \left(-\dfrac{J_2}{2\mathcal{T}}, \dfrac{J_1}{2\mathcal{T}} \right) \\
        &= \left(-\dfrac{J_2}{2\mathcal{T}}, \dfrac{J_1}{2\mathcal{T}} \right) \; .
    \end{align*}

    The leading quadratic coefficient for the energy growth for a vector $\Phi(x)=\big(\varphi(x),\varphi(x+1)\big)$ for the initial wave $\varphi$ is 
    \begin{equation*}
        a(1,2) = (J_2 - J_1)^2 \int_0^1 (x^2 + (x-1)^2) \bar{\varphi} \varphi \; dx >0 \; .
    \end{equation*}
\end{example}

\vspace{10pt}

Finally, we conclude this section with the proof of Theorem \ref{thm:quasienspec}.

\begin{proof}[Proof of Theorem \ref{thm:quasienspec}]
Since we deal with a basis of initial waves, we consider waves of the following form 
\begin{equation}\label{eq:ansatz}
    \psi_{j(x_0)} (x,0) := \sum_{n=0}^{q-1} C_n^j (x_0) \delta(x + x_0 + 2n/q)
\end{equation}
where $x_0\in(0,1)$ is a parameter, and $C_n^j(x_0)$'s are coefficients depending on $x_0$ that are to be determined. 

We study the one-period evolution of such $\psi$.

By \eqref{eq:simpfloq1}, 
\begin{align*}
    \psi_{j(x_0)} (x,\mathcal{T}) &= e^{-i J_1 \tilde{x}^2} \sum_{n=0}^{q-1} \gamma_n \psi_{j(x_0)}(x + 2n/q ,0)\\
    &= \alpha_0 (x) \sum_{n=0}^{q-1} \gamma_n \sum_{m=0}^{q-1} C_m^j (x_0) \delta(x + x_0 + 2(m+n)/q)\\
    &= \sum_{m,n=0}^{q-1} \gamma_n \alpha_{m+n}(x_0) C_m^j (x_0) \delta(x + x_0 + 2(m+n)/q).
\end{align*}

Then by \eqref{eq:2floquet}, 
\begin{align}
    &\quad \quad \psi_{j(x_0)}(x,2\mathcal{T})\\
    &= e^{i J_2 \tilde{x}^2} \sum_{l=0}^{q-1} \gamma_l \psi_{j(x_0)}(x + 2l/q , \mathcal{T}) \nonumber\\
    &= \beta_0(x) \sum_{l=0}^{q-1} \gamma_l \sum_{m,n=0}^{q-1} \gamma_n \alpha_{m+n}(x_0) C_m^j (x_0) \delta(x + x_0 + 2(m+n+l)/q) \nonumber\\
    &= \sum_{l,m,n=0}^{q-1} \gamma_l \gamma_n \alpha_{m+n}(x_0) \beta_{m+n+l}(x_0) C_m^j (x_0) \delta(x + x_0 + 2(m+n+l)/q) \label{eq:ansatzev}. 
\end{align}

We claim that there exist coefficients $C_n^j(\cdot)$ such that \eqref{eq:ansatz} fulfills the requirement \eqref{eq:quasienergydef} in the quantum Floquet theory, i.e. 
\begin{align}
    \psi_{j(x_0)} (x,2\mathcal{T}) 
    &= e^{-2i\mathcal{T} \rho_j(x_0)} \psi_{j(x_0)} (x,0) \nonumber \\
    & \buildrel \eqref{eq:ansatz} \over = e^{-2i\mathcal{T} \rho_j(x_0)} \sum_{k=0}^{q-1} C_n^k (x_0) \delta(x + x_0 + 2k/q) \label{eq:ansatzreq}.
\end{align}

Indeed, equating \eqref{eq:ansatzev} and \eqref{eq:ansatzreq} in index $k$, we obtain 
\begin{align*}
    e^{-2i\mathcal{T} \rho_j(x_0)} C_n^k (x_0) 
    &= \sum_{l+m+n=k \pmod q} \gamma_l \gamma_n \alpha_{m+n}(x_0) \beta_{m+n+l}(x_0) C_m^j (x_0)\\
    &= \sum_{l+m+n=k \pmod q} \gamma_l \gamma_n \alpha_{m+n}(x_0) \beta_{k}(x_0) C_m^j (x_0)\\
    &= \sum_{m,r=0}^{q-1} \gamma_{k-r} \gamma_{r-m} \alpha_{r}(x_0) \beta_{k}(x_0) C_m^j (x_0)\\
    &= \sum_{m=0}^{q-1} \sum_{r=0}^{q-1} \underbrace{\beta_{k}(x_0) \gamma_{k-r}}_{R_{kr}} \; \underbrace{\alpha_{r}(x_0) \gamma_{r-m}}_{S_{rm}}  C_m^j (x_0)\\
    &= \sum_{m=0}^{q-1} (RS)_{km}  C_m^j (x_0)\\
    &= \big((RS)(x_0)  \mathbf{C}^j (x_0) \big)_k
\end{align*}
where $\mathbf{C}^j(x_0) := (C_0^j (x_0),\cdots,C_{q-1}^j (x_0))$. 
Therefore the claim \eqref{eq:ansatzreq} holds for $\mathbf{C}^j (x_0)$ being an eigenvector of the resonant Floquet matrix $RS(x_0)$ with respect to the eigenvalue $\lambda_j (x_0)$. 

Consequently, we obtain from 
\begin{equation*}
    \lambda_j (x_0) = e^{i\xi_j(x_0)} = e^{-2i\mathcal{T}\rho_j(x_0)}
\end{equation*}
the formula of the quasi-energy 
\begin{equation*}
    \rho_j(x_0) = -\dfrac{1}{2\mathcal{T}} \xi_j(x_0) \; .
\end{equation*}
\end{proof}

\section{Concluding remarks}\label{sec:conc}

In this paper, we have discussed the classical piecewise linear Fermi-Ulam accelerators and its quantization. We have provided precise statements regarding the choice of parameters to produce resonance. We showed that under resonance the behavior of the quantized accelerator differs substantially from its classical origin in the sense that the classical accelerator exhibits typical recurrence/non-escaping while the quantum version enjoys quadratic energy growth in general. We also described a procedure, explicit and easily achievable, to locate the escaping orbits (though exceptionally rare) in the classical Fermi-Ulam accelerators, which in particular include Ulam's very original proposal \cite{Ulam61} and the linearly escaping orbits therein \cite{Zha98} in the existing literature, and hence provides a complete (modulo a null set) answer to Ulam's original question. 

An immediate follow-up question would be: what can we say if we lose resonance?

First of all, we emphasize that the derivation of many essential tools in this paper, such as the adiabatic normals, invariant circles, stopping the wall transformation and the Floquet operators, do not rely on the particular choice of parameters $(A,B,T)$, and hence they remain valid in the non-resonant cases as well.

However, there are several challenges to overcome. In the classical setting, we lose the skew product structure for non-resonant parameters, and the dynamics on each invariant circle is an infinite interval exchange map on $\mathbb{R}$, which, unfortunately, is a widely open subject in dynamics by itself. In the quantum setting, we believe that one should expect the absolutely continuous quasi-energy spectrum to degenerate into singularly continuous or pure point as the resonance numbers $\dfrac{\pi}{2}\dfrac{T}{AB}=\dfrac{p}{q}$ approach Diophantine or Liouville numbers, but linear algebra problems involving large matrices are in general very difficult theoretically and numerically.

\section*{Acknowledgment}

This research was partially supported by National Key R\&D Program of China No. 2024YFA1015100. C. Dong was also grateful for the support by Nankai Zhide Foundation, ``the Fundamental Research Funds for the Central Universities" No. 100-63233106, 100-63243066, and 100-63253093.

J. Zhou gratefully acknowledges Chern Institute of Mathematics where most of the work was done during her visits, and the warm hospitality and excellent working conditions are greatly appreciated. J. Zhou was partially supported by PRC Ministry of Education Research Grant RCXMA23008 and Guangdong Research Grant No.2023QN10X122.


\bibliography{originalFUM}
\bibliographystyle{plain}

\end{document}